\documentclass[a4paper]{article}

\usepackage{amsmath}
\usepackage{amsfonts}
\usepackage[dvipdfmx]{graphicx}
\usepackage{float}
\usepackage{alltt}
\usepackage{color}
\usepackage[top=30truemm,bottom=30truemm,left=25truemm,right=25truemm]{geometry}
\usepackage{url}

\newtheorem{thm}{Theorem}
\newtheorem{prop}{Proposition}
\newtheorem{cor}{Corollary}
\newtheorem{lemma}{Lemma}
\newtheorem{proof}{Proof}
\renewcommand{\theproof}{}
\newtheorem{definition}{Definition}
\newtheorem{algo}{Algorithm}
\newtheorem{example}{Example}
\newtheorem{remark}{Remark}

\newtheorem{benchmark}{Benchmark Problem}

\def\qed{\hfill $\Box$}
\def\pd#1{\partial_{#1}}

\def\almostsurely{\mathrm{a.s.}} 
\def\Prob{\mathbf{P}} 
\def\E#1{\mathbf E\left(#1\right)}

\def\({\left(}
\def\){\right)}
\def\ObservedData{\mathcal D}
\def\because{}
\def\im{\mathrm{Im}}
\def\mytheta{c}
\def\gi{g}
\def\gj{h}

\title{Holonomic Gradient Method for Two Way Contingency Tables}
\author{Yoshihito Tachibana, 
Yoshiaki Goto,
Tamio Koyama,
Nobuki Takayama}
\date{July 5, 2019 }

\begin{document}
\maketitle

\begin{abstract}
  The holonomic gradient method gives an algorithm
  to efficiently and accurately evaluate normalizing constants
  and their derivatives.
  We apply the holonomic gradient method
  in the case of the conditional Poisson or multinomial distribution on
  two way contingency tables.
  We utilize the modular method in computer algebra or some other tricks for
  an efficient and exact evaluation,
  and we compare them and discuss on their implementation.
  We also discuss on a theoretical aspect of 
  the distribution
  from the viewpoint of the conditional maximum likelihood estimation.
  We decompose parameters of interest and nuisance parameters in terms of
  sigma algebras for general two way contingency tables with arbitrary
  zero cell patterns.
\end{abstract}

\section{Introduction}

The holonomic gradient method (HGM) proposed in \cite{HGM}
provides an algorithm to efficiently and accurately 
evaluate normalizing constants and their derivatives.
This algorithm utilizes holonomic differential equations or
holonomic difference equations.
Y. Goto and K. Matsumoto \cite{MG} determined a system of difference equations
for the hypergeometric system of type $(k,n)$.
The normalizing constant
of the conditional Poisson or multinomial distribution on
two-way contingency tables is a polynomial solution
of this hypergeometric system.
Thus, we can apply these difference equations to exactly evaluate the normalizing constant
and its derivatives by HGM.
However, there is a difficulty:
numerical evaluation errors, incurred by repeatedly applying these difference equations
or recurrence relations, increase rapidly 
if we use floating point number arithmetic.
Accordingly, we evaluate the normalizing constant by exact rational arithmetic.
However, in general, exact evaluation is slow.
The modular method in computer algebra 
(see, e.g., \cite{MM1}, \cite{MM2})
has been used for efficient 
and exact evaluation over the field of rational numbers. 
We apply the modular method or some other tricks to our evaluation procedure.
We compare these methods and explore implementation of these algorithms 
in Sections \ref{sec:contiguity} and \ref{sec:modular_and_tricks}.

We then turn from computation to a theoretical question before presenting
statistical applications.  
An interesting application of the evaluation of the normalizing
constant is the conditional maximum likelihood estimation (CMLE)
of parameters of interest with fixed marginal sums. 
Broadly speaking, the parameters of interest in this case are
(generalized) odds ratios.
However, we could not identify a rigorous formulation on parameters of interest 
for contingency tables with zero cells in the literature.
In Sections \ref{sec:sufficient} and \ref{sec:application},
we introduce ${\cal A}$-distributions as a conditional distribution.
The conditional Poisson or multinomial distribution on contingency tables
with fixed marginal sums
is a special and important case of ${\cal A}$-distributions.
We will decompose parameters of interest
and nuisance parameters in terms of $\sigma$-algebras.
We note that 
the conditional distribution of a statistic
given the occurrence of a sufficient statistic of a nuisance parameter
does not depend on the value of the nuisance parameter.
Hence, by the conditional distribution,
we can estimate the parameter of interest
without being affected by the nuisance parameter.

Finally, we apply our method to a CMLE problem for contingency tables.
This problem is discussed in \cite{ogawa} for the case of
$2 \times n$ contingency tables and the work presented here 
generalizes this to two-way contingency tables 
of any size and with any pattern of zero cells.

\section{Two Way Contingency Tables} \label{sec:two}
 We introduce our notation for contingency tables
and review how the normalizing constant for a conditional distribution
is expressed by a hypergeometric polynomial of type $(k,n)$.
There are several salient references on contingency tables. 
Among them, we will refer to \cite{agresti} and \cite[Chap 4]{dojo} herein.
\subsection{$r_1\times r_2$ Contingency Table}
\begin{definition}[$r_1 \times r_2$ (2 way) contingency table]
An $r_1 \times r_2$ matrix with non-negative integer entries
is called an
$r_1 \times r_2$ {\it contingency table}\/.  
For a contingency table $u=(u_{ij})$, we define the {\it row sum vector} by 
$\beta^{r}=\left(\sum_{j}u_{1j},\cdots ,\sum_{j}u_{r_{1}j}\right)^{T}$, 
and the {\it column sum vector} by
$\beta ^{c}=\left(\sum_{i}u_{i1},\cdots ,\sum_{i}u_{ir_{2}}\right)^{T}$.
A contingency table $u$ is also written as a column vector of length 
$r_1 \times r_2$, denoted by $u^{f}$. 
The column vector obtained by joining 
$\beta ^{r}$ and $\beta ^{c}$
is denoted by $\beta$,
which is called the {\it row column sum vector} 
or the {\it marginal sum vector}\/.
\end{definition}

\begin{example}[$2 \times 3$ contingency table and the row sum and the column sum] \rm

For the $2 \times 3$ contingency table
       $u=\left(\begin{array}{ccc}5&3&6\\7&2&4\end{array}\right)$
the row sum vector and the column sum vector are 
$$\beta^{r}=\left(\begin{array}{c}5+3+6=14\\7+2+4=13\end{array}\right),\
  \beta^{c}=\left(\begin{array}{c}5+7=12\\3+2=5\\6+4=10\end{array}\right).
$$
The corresponding vector expressions of
$u^{f}$ and $\beta$
are 
$$u^{f}=\left(\begin{array}{cccccc}5&3&6&7&2&4\end{array}\right)^{T},\
	       \beta=\left(\begin{array}{ccccc}14&13&12&5&10\end{array}\right)^{T}.$$
\end{example}

We fix $p=(p_{ij})\in \mathbb{R}_{>0}^{r_1\times r_2},\ N\in \mathbb{N}_{0}$
and consider the multinomial distribution 
$$\frac{N!p^{u}}{u!|p|^{N}},\ p^{u}=\prod _{i,j}p_{ij}^{u_{ij}},\
u!=\prod_{i,j}u_{ij}!$$
on contingency tables satisfying $|u|=\sum_{i,j}u_{ij}=N$.
The conditional distribution obtained by fixing
the row sum vector $\beta ^{r}$ and the column sum vector 
$\beta^{c}$ is
\begin{equation}
\frac{p^{u}}{u!Z(\beta ; p)},\quad\quad
Z(\beta ;p)=\sum_{Au^{f}=\beta,\
u\in \mathbb{N}_{0}^{r_1 \times r_2}}\frac{p^{u}}{u!}.
\end{equation}
Here, the polynomial
$Z(\beta ;p)$
is the normalizing constant of this conditional distribution.
The matrix $A$ satisfies the following conditions:
(1) entries are $0$ or $1$; 
(2) $Au^f$ is the marginal sum vector
(see Example \ref{example:2}).
The expectation of the $u$-value at $(i,j)$ 
of this conditional distribution
is equal to
\begin{equation}
   E[U_{ij}] = p_{ij}\frac{\partial \log Z}{\partial p_{ij}} .
\end{equation}
Exact evaluation of the conditional probability of getting 
a contingency table $u$ and evaluation of the expectation is reduced to
the evaluation of the normalizing constant $Z$ and its derivatives.
For given rational numbers $p_{ij}$, we provide an efficient and exact method to evaluate $Z$ and its derivatives.

\begin{example}[example of $A$]\label{example:2} \rm
When
$u^f = \left(\begin{array}{cccccc}5&3&6&7&2&4\end{array}\right)^{T}$,
the matrix $A$ is 
$$A=\left(\begin{array}{cccccc}
1&1&1&0&0&0\\
0&0&0&1&1&1\\
1&0&0&1&0&0\\
0&1&0&0&1&0\\
0&0&1&0&0&1\\
\end{array}\right)$$
and we have
$$Au^{f}=\left(\begin{array}{cccccc}
1&1&1&0&0&0\\
0&0&0&1&1&1\\
1&0&0&1&0&0\\
0&1&0&0&1&0\\
0&0&1&0&0&1\\
\end{array}\right)\left(\begin{array}{c}
5\\
3\\
6\\
7\\
2\\
4\end{array}\right)=\left(\begin{array}{c}
14\\
13\\
12\\
5\\
10\end{array}
\right)=\beta.$$
\end{example}
\begin{example} \rm
We consider $2\times 2$ contingency tables with
the marginal sum vector
$\beta = \left(\begin{array}{cccc}5&7&8&4\end{array}\right)^{T}$.
All contingency tables $u$ satisfying
$Au^{f}=\beta$
are
$$\left(\begin{array}{cc}5&0\\3&4\end{array}\right),\ 
\left(\begin{array}{cc}4&1\\4&3\end{array}\right),\ 
\left(\begin{array}{cc}3&2\\5&2\end{array}\right),\ 
\left(\begin{array}{cc}2&3\\6&1\end{array}\right),\ 
\left(\begin{array}{cc}1&4\\7&0\end{array}\right).
$$
These $u$ are written as
$$\left(\begin{array}{cc}5&0\\3&4\end{array}\right)+i\left(\begin{array}{cc}-1&1\\1&-1\end{array}\right),\
	    (i=0,1,2,3,4).$$
\end{example}

\section{The Normalizing Constant of $2 \times 2$ Tables} \label{sec:normalizing}
It is known that the normalizing constant 
for the conditional distribution for 
$r_1 \times r_2$
tables is $A$-hypergeometric polynomial (see, e.g.,  \cite[Section 6.13]{dojo}).
We will illustrate this correspondence for $2 \times 2$ 
contingency tables.

Consider the marginal sum vector
$\beta
=\left(u_{11},u_{21}+u_{22},u_{11}+u_{21},u_{22}\right)$
with $u_{ij} \geq 0$.
The $2 \times 2$ contingency tables with the marginal sum vector
$\beta$ are
$$u=\left(\begin{array}{cc}u_{11}&0\\u_{21}&u_{22}\end{array}\right)+i\left(\begin{array}{cc}-1&1\\1&-1\end{array}\right),\
(i=0,1,2,\cdots ,n).$$
Here, we have $n=\min\{u_{11},u_{22}\}$. 
The normalizing constant is
\begin{eqnarray*}
Z(\beta ;p)&=&\sum
 _{i=0}^{n}\frac{p_{11}^{u_{11}-i}p_{12}^{i}p_{21}^{u_{21}+i}p_{22}^{u_{22}-i}}{(u_{11}-i)!(i)!(u_{21}+i)!(u_{22}-i)!}\\
&=&\frac{p_{11}^{u_{11}}p_{21}^{u_{21}}p_{22}^{u_{22}}}{u_{11}!u_{21}!u_{22}!}\sum _{i=0}^{n}\frac{(-u_{11})_{i}(-u_{22})_{i}}{(u_{21}+1)_{i}(1)_{i}}\left(\frac{p_{12}p_{21}}{p_{11}p_{22}}\right)^{i},
\end{eqnarray*}
where $(a)_i = a(a+1) \cdots (a+i-1)$.
Then, it can be expressed in terms of the Gauss hypergeometric function
$$
{}_2
F_{1}(a,b,c;x)=\sum_{i=0}^{\infty}\frac{(a)_{i}(b)_{i}}{(c)_{i}(1)_{i}}x^{i}.$$
Note that when $a,b\in \mathbb{Z}_{\leq 0}$, it is a polynomial.
The normalizing constant can also be expressed in terms of ${}_2F_1$ for other types
of marginal sum vectors.
A consequence of this observation is that we can utilize
several formulae of the hypergeometric function to evaluate the normalizing
constant.

\section{Contiguity relation} \label{sec:contiguity}
In the previous section, we expressed the normalizing constant
for $2 \times 2$ contingency tables with a fixed marginal sum vector
in terms of the Gauss hypergeometric function.
For $r_1 \times r_2$ contingency tables,
the normalizing constant with a fixed marginal sum vector
can be expressed in terms of the Aomoto-Gel'fand 
hypergeometric function of type $(r_1, r_1+r_2)$
\cite{TKT}
(the function ${}_2F_{1}$ is of type $(2,4)$). 
This hypergeometric function is also called the
$A$-hypergeometric function for the product of the $(r_1-1)$-simplex
and $(r_2-1)$-simplex.
The difference holonomic gradient method for these hypergeometric functions
utilizes contiguity relations.
We illustrate this for the case of the Gauss hypergeometric function; for the general case, see \cite{MG}.
\begin{example}[the case of ${}_2F_1$] \label{example:4}  \rm
Put 
$f(a)={}_2F_1(a,b,c;x)$
and
$$F(a)=\left(\begin{array}{c}f(a)\\ \theta_xf(a)\end{array}\right),\ M(a)=\frac{1}{a-c+1}\left(\begin{array}{cc}bx+a-c+1&x-1\\
			   -abx&a(1-x)\end{array}\right),$$
where $\theta_x $ is the Euler operator $ x\partial _x$.
Then, we have 
\begin{equation} \label{eq:2f1downstep}
F(a)=M(a)F(a+1).
\end{equation}
Now, note the following relations:
\begin{eqnarray}
\frac{1}{a}(a+\theta _{x})\bullet f(a)&=&f(a+1)\label{2f1_1},\\
\big( \theta _{x}(c-1+\theta _{x})-x(a+\theta_{x})(b+\theta_{x}) \big)\bullet f(a)&=&0.\label{2f1_2}
\end{eqnarray}
The first relation can be shown from the series expansion and the second relation
is the Gauss hypergeometric differential equation.
It follows from (\ref{2f1_1}), (\ref{2f1_2}) that we have
\begin{eqnarray*}
\frac{1}{a}(a+\theta _{x})\bullet F(a)&=&F(a+1),\\
\theta_x F(a) &=&\left(\begin{array}{cc}
0&1\\
\frac{abx}{1-x}&\frac{ax+bx-c+1}{1-x}\end{array}\right)F(a)\\
&=&A(a)F(a).
\end{eqnarray*}
Next, we have (\ref{eq:2f1downstep}) as
\begin{eqnarray*}
\frac{1}{a}(a+\theta _x)\bullet F(a)&=&\frac{1}{a}(aE+A(a))F(a),\\
F(a)&=&\left(\frac{1}{a}(aE+A(a))\right)^{-1}F(a+1)\\
&=&M(a)F(a+1),
\end{eqnarray*}
where $E$ is the identity matrix.
\end{example}

A relation like $F(a) = M(a)F(a+1) $
is called a {\it contiguity relation}\/.
In \cite{MG}, the vector valued function $F(a)$ 
is called the {\it Gauss-Manin vector}\/.

There are several algorithms to obtain contiguity relations
\cite{T1989}, \cite{Oshima}, \cite{OT}, \cite{MG}.
Among them, we choose to use the method of twisted cohomology groups
given in \cite{MG},
because it is the most efficient method for the case of two-way
contingency tables.

We briefly summarize the method given in \cite{MG}.
Consider the hypergeometric series 
$f(\alpha ;x)$ of type $(r_1,r_1+r_2)$.
Here, the parameter $\alpha =(\alpha_1,\ldots ,\alpha_{r_1+r_2-1})$ stands for
the marginal sum vector $\beta$
and  
the variable $x=(x_{ij})_{1\leq i\leq r_1-1, 1\leq j\leq r_2-1}$ stands for $p$.
It follows from the twisted cohomology group (a vector space spanned by 
equivalence classes of differential forms)
associated to the integral representation of $f$
that the contiguity relation 
for 
  $\alpha_i \to \alpha_i +1$ 
can be obtained as follows.

We consider the twisted cohomology group $H$ (resp. $H'$) standing for
the function $f(\alpha;x)$ (resp. $f(\alpha;x)|_{\alpha_i \to \alpha_i +1}$). 
Both twisted cohomology groups are of dimension  
$\displaystyle r=\binom{r_1+r_2-2}{r_1-1}$.
We take a basis $\varphi_1, \ldots, \varphi_r$ of $H$ such that 
the ``integral'' of $(\varphi_1 ,\ldots ,\varphi_r )^T$ gives 
a constant multiple of the Gauss-Manin vector
$$
F(\alpha ;x)=(f(\alpha ;x), \delta^{(2)} \bullet f(\alpha ;x) ,\ldots ,
\delta^{(r)}\bullet f(\alpha ;x) )^T ,
$$
where $\delta^{(i)}$ is some differential operator with respect to $x=(x_{ij})$. 
There exist a basis $\varphi'_1, \ldots, \varphi'_r$ of $H'$ and 
a linear map $\mathcal{U}_i :H' \to H$ such that 
the integral of 
$\bigl( \mathcal{U}_i(\varphi'_1 ),\ldots ,\mathcal{U}_i(\varphi'_r) \bigr)^T$
gives a constant multiple of the shifted Gauss-Manin vector
$F(\alpha ;x)|_{\alpha_i \to \alpha_i +1}$. 
Let $U_i (\alpha ;x)$ be a representation matrix of $\mathcal{U}_i$ with respect to the bases 
$\{ \varphi'_i \}$ and $\{ \varphi_j \}$:
$$
\bigl( \mathcal{U}_i(\varphi'_1 ),\ldots ,\mathcal{U}_i(\varphi'_r) \bigr)^T
=U_i(\alpha ;x) \cdot (\varphi_1 ,\ldots ,\varphi_r )^T .
$$
Integrating both sides, 
we thus obtain the
contiguity relation
$$
F(\alpha ;x)|_{\alpha_i \to \alpha_i +1}=\tilde{U}_i (\alpha ;x)F(\alpha ;x),
$$
where $\tilde{U}_i$ is a constant multiple of $U_i$.
It turns out that the representation matrix $U_i$ can be expressed in terms of
a simple diagonal matrix and base transformation matrices
which can be obtained by evaluating intersection numbers among
differential forms.
The contiguity relation for $\alpha_i \to \alpha_i -1$ can be derived
analogously.
For more details, see \cite{MG}. 
Here, we illustrate this method in the case of ${}_2F_1$.

\begin{example}[the case of $_2 F_1$ ($r_1=r_2=2$, $r=2$)] \label{example:5} \rm
 For the parameter  $(a,b,c)$ of $_2 F_1$,
we put
    $$
    (\alpha_1 ,\alpha_2 ,\alpha_3)=(b,-a,c-b-1).
    $$
Here, we set $\alpha_0 =-\alpha_1 -\alpha_2 -\alpha_3=a-c+1$ for convenience.
Since the move $a+1 \to a$ corresponds to
    $\alpha_2-1 \to \alpha_2$ (and $\alpha_0+1 \to \alpha_0$)
in the new parametrization,
the matrix $M(a)$ in Example \ref{example:4} 
stands for $U_2 (\alpha ;x)$.
The representation matrix 
$U_2$ has the following decomposition.\footnote{see the appendix (Section \ref{sec:appendix}) for more details.}
    \begin{align*}
      U_2 =& \frac{\alpha_1 (\alpha_2 -1)}{\alpha_3} 
      \begin{pmatrix}
        \frac{1}{\alpha_0}+\frac{1}{\alpha_1} & \frac{1}{\alpha_0} \\
        \frac{1}{\alpha_0} & \frac{1}{\alpha_0}+\frac{1}{\alpha_2}
      \end{pmatrix}
      \begin{pmatrix}
        \alpha_1 & -\alpha_1 \\ 0 & -\alpha_2
      \end{pmatrix} \\
      & \qquad \cdot 
      \begin{pmatrix}
        1 & 0 \\ 0 &1-x
      \end{pmatrix}
      \begin{pmatrix}
        \frac{1}{\alpha_0+1}+\frac{1}{\alpha_1} & \frac{1}{\alpha_0+1} \\
        \frac{1}{\alpha_0+1} & \frac{1}{\alpha_0+1}
      \end{pmatrix}
      \begin{pmatrix}
        \frac{\alpha_1 +\alpha_3}{\alpha_2 -1} & 1 \\ 
        1 & \frac{\alpha_2 -1+\alpha_3}{\alpha_1}
      \end{pmatrix}.
    \end{align*}
The matrices except the diagonal matrix ${\rm diag}(1,1-x)$ 
are expressed by intersection numbers.
Since we have $\delta^{(2)}=\frac{1}{\alpha_2}\theta_x$,
the matrix $U_2$ has a small difference with $M(a)$ in Example
\ref{example:4} and we obtain $M(a)$ by adjusting the scale factor $1/\alpha_2$ of $\theta_x$.
\end{example}

By the contiguity relation, we can evaluate the normalizing constant $Z$
and its derivatives.
Let us explain the procedure for the case of ${}_2 F_{1}$.

Suppose $a\in \mathbb{Z}_{<-1}$.
By the contiguity relation (\ref{eq:2f1downstep}), we have
\begin{eqnarray}
F(a)&=&M(a)F(a+1) \nonumber \\
&=&M(a)M(a+1)F(a+2) \nonumber \\
&\vdots& \nonumber \\
&=&M(a)M(a+1)\cdots M(-2)F(-1). \label{eq:composite_of_linear_transformation}
\end{eqnarray}
Then, we can obtain the value of $F(a)$ from the initial value 
$F(-1)=(1-\frac{b}{c}x,-\frac{b}{c}x)^T$ 
by applying linear transformations.
Values of the normalizing constant and its derivatives can be obtained
from $F(a)$ with the differential equation for the Gauss hypergeometric 
function.
This method is called the {\it difference holonomic gradient method}
(difference HGM) and can be generalized to the case of $r_1 \times r_2$
contingency tables 
with the Gauss-Manin vector and contiguity relations 
given in \cite{MG}.

We note that a naive evaluation of the polynomial $Z$ is very slow.
For example, the polynomial $Z$ of the $2 \times 5$ contingency table 
with the row sum $(4n,5n)$, the column sum $(5n,n,n,n,n)$
and 
$p = \left(\begin{array}{ccccc}
1 & 1/2 & 1/3 & 1/5 & 1/7 \\
1 & 1 & 1& 1& 1 \\
\end{array}\right)$
can be
expressed in terms of the Lauricella function 
$F_D(-4n;-n,-n,-n,-n; n+1; 1/2,1/3,1/5,1/7)$ of $4$ variables (see, e.g.,  \cite{goto-2017}).
The number of terms is $O(n^4)$.
Here is a comparison of the naive summation of $F_D$ and our HGM implementation discussed
in the next section. \\[0.5ex]
\begin{tabular}{l|rrr}
$n$             & 20   & 30    & 40 \\ \hline
Naive summation (in seconds) & 16.0 & 111.7 & 456.6 \\
HGM             (in seconds) & 0.28 & 0.276 & 0.284 \\
\end{tabular} \\[0.5ex]
Thus, the HGM is worth researching.

We briefly introduce an algorithm of difference HGM for $r_1 \times r_2$ contingency tables. 
The following algorithm computes the Gauss-Manin vector $F(\beta;p)$ which is 
essentially same as $F(\alpha ;x)$ in the above 
(for the correspondence between $(\beta ;p)$ and $(\alpha ;x)$, see \cite[Proposition 7.1]{MG}).  
In fact, we give improvement of Step 2--4 of \cite[Algorithm 7.8]{MG}. 

\begin{algo}[A modified version of {\cite[Step 1--4 of Algorithm 7.8]{MG}}] \label{algorithm:contiguity}~

Input: $\beta =(\beta^{(1)}_1,\dots ,\beta^{(1)}_{r_1} ;\beta^{(2)}_1 ,\dots ,\beta^{(2)}_{r_2})$: 
a marginal sum vector, 
$p=(p_{ij}) \in \mathbb{Q}_{>0}^{r_1\times r_2}$: probabilities of the cells.

Output: the Gauss-Manin vector $F(\beta ;p)$ (which is a vector of size $r={r_1+r_2-2 \choose r_1-1}$).

\begin{enumerate}
\item Set $B_0 =(1,\dots ,1,\beta^{(1)}_1+\cdots +\beta^{(1)}_{r_1}-r_1+1;
  \beta^{(2)}_1 ,\dots ,\beta^{(2)}_{r_2})$. 
  Compute $F(B_0 ;p)$ by the definition. 
  (In this case, the normalizing constant $Z(B_0 ;p)$ is a polynomial of small degree, and hence
  the Gauss-Manin vector $F(B_0 ;p)$ is easily computed.) 
\item For $k=1,\dots ,r_1-1$, define $B_k$ inductively as 
  $B_k =B_{k-1} +(\beta^{(1)}_k -1)\cdot \delta_k$, 
  where 
  $$
  \delta_k= (0 ,\dots ,0,\underset{k\text{-}{\rm th}}{1},0,\dots ,0,-1;0,\dots ,0) 
  $$
  (note that $B_{r_1-1}$ is $\beta$).  
  Evaluate the contiguity matrices $C_k (t)$ that satisfy 
  \begin{align*}
    F(B_{k-1}+(T+1)\delta_k ;p)=C_k(T)\cdot F(B_{k-1}+T\delta_k ;p),\quad 
    T=0,1,\dots ,\beta^{(1)}_k -2. 
  \end{align*}
  Here, $t$ is an indeterminate and each entry of $C_k (t)$ is 
  an element of $\mathbb{Q}(t)$. 
\item For $k=1,\dots ,r_1-1$, compute $F(B_k ;p)$ inductively as 
  \begin{align}  \label{GLT}
    F(B_k;p) =C_k(\beta^{(1)}_k -2) \cdots C_k(1) C_k(0)F(B_{k-1};p). 
  \end{align}
\item Return $F(B_{r_1-1};p)$.
\end{enumerate}
\end{algo}
By using $F(\beta ;p)$, we can compute the normalizing constant $Z(\beta ;p)$ and 
the expectations $E[U_{ij}]$ (see \cite[Step 5--7 of Algorithm 7.8]{MG}). 

\begin{example}[cf. {\cite[Example 7.10]{MG}}]
  We consider $3\times 3$ contingency tables whose marginal sum vector is 
  $\beta =(2,3,3;1,3,4)$. 
  In this case, the Gauss-Manin vector is of size ${3+3-2 \choose 3-1}=6$. 
  \begin{enumerate}
  \item We set $B_0 =(1,1,6;1,3,4)$, and compute $F(B_0;p)$ by the definition. 
    In this case, the normalizing constant $Z(B_0;p)$ has only eight terms. 
  \item We set 
    $B_1 =(2,1,5;1,3,4)$, 
    $B_2 =(2,3,3;1,3,4) (=\beta)$. 
    By using notations in \cite{MG}, we put 
    \begin{align*}
      C_1 (t)=U_1^{-1}(-5+t,-2-t,-1,3,4,1;x),\quad 
      C_2 (t)=U_2^{-1}(-4+t,-2,-2-t,3,4,1;x) .
    \end{align*}
    Here, $x\in \mathbb{Q}^{(r_1-1)\times (r_2-1)}$ is defined from $p$. 
    We have
    \begin{align*}
      &C_1(0) F(1,1,6;1,3,4;p)= F(2,1,5;1,3,4;p) ,\\
      &C_2(0) F(2,1,5;1,3,4;p)= F(2,2,4;1,3,4;p) ,\quad 
       C_2(1) F(2,2,4;1,3,4;p)= F(2,3,3;1,3,4;p) .
    \end{align*}
  \item We compute the product
    \begin{align*}
      C_2(1)C_2(0)C_1(0)F(B_0;p)
      &=C_2(1)C_2(0)C_1(0)F(1,1,6;1,3,4;p) \\
      &=C_2(1)C_2(0)F(2,1,5;1,3,4;p) \ (=C_2(1)C_2(0)F(B_1;p)) \\
      &=C_2(1)F(2,2,4;1,3,4;p)  \\
      &=C_2(1)F(2,3,3;1,3,4;p) \ (=F(B_2;p)) .
    \end{align*}
  \item We obtain the Gauss-Manin vector $F(B_2;p)=F(\beta ;p)$. 
  \end{enumerate}
  For example, when 
  $p=\left(
    \begin{array}{ccc}
      1 & 1/2 & 1/3 \\
      1 & 1/5 & 1/7 \\
      1 &1 &1 \\
    \end{array}
  \right)$, 
  the $6\times 6$ matrix 
  $C_2 (t)$ is given as follows\footnote{
    It is obtained by our program \texttt{gtt\_ekn3} as \\
    \texttt{gtt\_ekn3.downAlpha3(2,2,2 | arule=gtt\_ekn3.alphaRule\_num([-5+t,-2,-1-t,3,4,1],2,2),} \\
    \phantom{\texttt{gtt\_ekn3.downAlpha3(2,2,2 |}}
    \texttt{xrule=gtt\_ekn3.xRule\_num([[1,1/2,1/3],[1,1/5,1/7],[1,1,1]],2,2))}. 
  }. 
\begin{align*}
  C_2 (t)=
  &\left(
  \begin{array}{cccccc}
    \frac{-( 35  {t}+ 29)} {35( {t}+ 2)}
    &  \frac{ 12} {5 (  {t}+ 2)}
    &  \frac{ 24} {7 (  {t}+ 2)}
    &  \frac{ - 12} {5 (  {t}+ 2)}
    &  \frac{ - 24} {7 (  {t}+ 2)}
    & 0 \cr
    \frac{ 1}{ 5}
    & \frac{-1}{ 5}
    & 0
    & \frac{ 1}{ 5}
    & 0
    & 0 \cr
    \frac{ 1}{ 7}
    & 0
    & \frac{ -1}{ 7}
    & 0
    & \frac{ 1}{ 7}
    & 0 \cr
    \frac{ - 8} { 5(  {t}+ 2)}
    &  \frac{ 8} { 5(  {t}+ 2)}
    & 0
    &  \frac{   21  {t}- 73} {35 (  {t}+ 2)}
    &  \frac{ - 88} { 35(  {t}+ 2)}
    &  \frac{ 88} {35 (  {t}+ 2)} \cr
    \frac{ - 6} {7 (  {t}+ 2)}
    & 0
    &  \frac{ 6} {7 (  {t}+ 2)} 
    &  \frac{ - 33} {35 (  {t}+ 2)}
    &  \frac{ 10  {t}- 47} {35 (  {t}+ 2)}
    &  \frac{ - 33} { 35(  {t}+ 2)} \cr
    0& 0& 0& \frac{ - 1}{ 35}& \frac{ 1}{ 35}& \frac{ - 1}{ 35}
  \end{array}
  \right) .
\end{align*}
\end{example}

\begin{remark}
  The algorithm given in \cite{MG} requires more matrix multiplications 
  than Algorithm \ref{algorithm:contiguity}. 
  As \cite[Example 7.10]{MG}, the former algorithm computes the above $F(2,3,3;1,3,4;p)$ 
  by nine matrix multiplications (each `` $\mapsto$'' means one multiplication): 
  \begin{align*}
    &F(1,1,2;2,1,1;p) 
      \mapsto F(1,1,3;2,2,1;p) 
      \mapsto F(1,1,4;2,3,1;p) \\
    &\mapsto F(1,1,5;2,3,2;p) 
      \mapsto F(1,1,6;2,3,3;p) 
      \mapsto F(1,1,7;2,3,4;p) \\
    &\mapsto F(1,1,6;1,3,4;p) 
      \mapsto F(2,1,5;1,3,4;p) 
      \mapsto F(2,2,4;1,3,4;p) 
      \mapsto F(2,3,3;1,3,4;p) .
  \end{align*}
  On the other hand, 
  Algorithm \ref{algorithm:contiguity} needs only the last three steps. 
\end{remark}

We give the complexity to construct the matrix $C_k (t)$. 
The appendix (Section \ref{sec:appendix}) will help to follow the following argument.
By \cite[Theorem 5.3]{MG}, 
the matrix $U_k^{\pm 1}$ for the contiguity relation is the product of 
five matrices of size 
$r={r_1+r_2-2 \choose r_1-1} =\frac{(r_1+r_2-2)!}{(r_1-1)!(r_2-1)!}$: 
\begin{enumerate}
\item[(a)] one diagonal matrix whose entries are rational functions in $p$, 
\item[(b)] two intersection matrices whose entries are rational functions in $\beta$,
\item[(c)] two inverse matrices of intersection matrices
\end{enumerate}
(cf. Example \ref{example:5}). 
For $U_k^{-1}$, by substituting 
\begin{itemize}
\item $\beta^{(1)}_{k}$ and $\beta^{(1)}_{r_1}$ with certain polynomials in $t$ of degree $1$, 
\item the other $\beta^{(i)}_{j}$'s and $p$ with certain rational numbers, 
\end{itemize}
we obtain the matrix $C_k (t)$. 
By this construction and the formula for (a), (b), (c) in \cite{MG}, 
it turns out that when we construct $C_k (t)$, we treat 
rational functions in $t$ 
whose denominator and numerator are of degree at most $12$. 
As long as we have tried on a computer for cases $5 \times r_i$, $r_i \leq 12$,
the degrees of numerators and denominators are much smaller than $12$ and
no big number (large number so that FFT multiplication algorithms are used) 
appears in the matrix $C_k(t)$;
when we use the modular method, all numbers in the matrix are elements
in a finite field.
Thus, we assume in the following theorem that 
the complexity of arithmetics of polynomials in one variable is $O(1)$. 

\begin{thm}
Let $r_1, r_2 \geq 2$. 
Assume that the complexity of arithmetics is $O(1)$,
the complexities of multiplying two $n\times n$ matrices 
and evaluating the determinant of an $n \times n$ matrix are $O(n^{\omega})$ for some $2\leq \omega <3$. 
The complexity of obtaining the matrix $C_k (t)$ 
in Algorithm \ref{algorithm:contiguity} 
for $r_1 \times r_2$ contingency tables
is $O(r^{\omega})$, where $r = {{r_1+r_2-2} \choose {r_1-1}}$. 
Especially, it is 
\begin{enumerate}
\item $O(r_2^{\omega r_1})$ when $r_1$ is fixed,
\item $O(r_1^{\omega r_2})$ when $r_2$ is fixed, 
\item $O(2^{2\omega r_1})$ when $r_1=r_2$.
\end{enumerate}
\end{thm}

\begin{proof} \rm
  As explained later, the complexity to construct the above matrices 
  (a), (b) and (c) are 
  $O(r_1^{\omega} r)$,  
  $O(r_1^2 r^2)$ and  
  $O(r_1^2 r^2)$, respectively.  
  Since the size of each matrix is $r$, 
  the complexity of multiplication is $O(r^{\omega})$. 
  Thus, the complexity to obtain a contiguity relation is 
  $O(r^{\omega})+O(r_1^{\omega} r)+O(r_1^2 r^2)$. 
  Since $r$ is larger than $r_1^2$ in general, 
  the complexity is equal to $O(r^{\omega})$. 
  \begin{enumerate}
  \item We fix $r_1$ and assume $r_2 \gg r_1$. 
    By the Stirling formula $\log n! \sim n\log n -n$, 
    we have 
    \begin{align*}
      \log r 
      &\sim (r_1 + r_2)\log (r_1+r_2) -r_2 \log r_2 \\
      &=r_1 \log r_2 +r_1 \log \Big( 1+\frac{r_1}{r_2} \Big) +r_2 \log \Big( 1+\frac{r_1}{r_2} \Big)
        \sim r_1 \log r_2 .
    \end{align*}
    Then we obtain $r \sim r_2^{r_1}$ and 
    the complexity is $O(r_2^{\omega r_1})$. 
  \item Claim 2 can be obtained by a similar argument to Claim 1.
  \item If $r_1=r_2$, then by the Stirling formula, we have 
    \begin{align*}
      \log r \sim 2r_1 \log 2r_1 -2r_1 \log r_1 =2r_1 \log 2, 
    \end{align*}
    which implies $r\sim 2^{2r_1}$. 
    Thus, the complexity is $O(2^{2\omega r_1})$.
  \end{enumerate}
  Now, we explain the complexity of obtaining the matrices (a), (b), (c). 
  \begin{enumerate}
  \item[(a)] As \cite[Theorem 5.3]{MG}, each nonzero entry of 
    the diagonal matrix is the ratio of determinants of two $r_1 \times r_1$ matrices. 
    Thus the complexity of evaluation is $O(r_1^{\omega} r)$.  
  \item[(b)] The entries of intersection matrices are 
    intersection numbers of $(r_1 -1)$-th twisted cohomology groups, 
    which can be evaluated by the formula in \cite[Fact 3.2]{MG}. 
    The complexity of evaluating an intersection number by this formula is $O(r_1^2)$, 
    and hence the complexity of obtaining the intersection matrix is $O(r_1^2 r^2)$. 
  \item[(c)] By the proof of \cite[Proposition A.1]{MG}, 
    the inverse matrix of an intersection matrix is expressed as a product of
    two diagonal matrices and one intersection matrix. 
    The complexity of obtaining the diagonal matrices is $O(r_1 r)$, 
    since that of their nonzero entry is $O(r_1)$.
    Therefore, the complexity of obtaining the inverse matrix of the intersection matrix
    is dominated by the complexity $O(r_1^2 r^2)$ of obtaining the intersection matrix. 
  \end{enumerate}
\qed
\end{proof}

In this section we conducted a complexity analysis of the method for obtaining the contiguity
relation.
The theoretical complexity is of a polynomial order when $r_i$ is fixed
and our implementation shows that this step is efficient for 
small sized contingency tables.
However, 
a naive evaluation of the composition of linear transformations
(\ref{eq:composite_of_linear_transformation})
is slow, even for small contingency tables, because of large numbers  
when $|a|$ is large.

\section{Efficient Evaluation of a Composition of Linear Transformations}
\label{sec:modular_and_tricks}

To perform exact and efficient evaluations
by the difference HGM, 
we need a fast and exact evaluation of a composition of linear transformations
for vectors with rational number entries.
This problem has hitherto been explored and there are several implementations,
e.g., LINBOX \cite{linbox}.
For the purposes of empirical application, we study several methods
to evaluate the composition of linear transformations
such as (\ref{eq:composite_of_linear_transformation}) 
or (\ref{GLT}).
Our implementation is published as the package {\tt gtt\_ekn3}
for Risa/Asir \cite{risa-asir}.
The function names in this section are those in this package.

\subsection{Our Benchmark Problems}

In order to compare several methods, we will use the following 4 
benchmark problems.
The timing data are taken on a machine with
 \begin{center}
  \begin{tabular}{ll}
   CPU & Intel(R) Xeon(R) CPU E5-4650 2.70GHz\\
   the number of CPU's & 32\\
   the number of cores & 8\\
   OS & Debian 9.8\\
   memory & 256GB\\
   software system & Risa/Asir (2018) version 20190328 with GMP \cite{GMP}
  \end{tabular}
 \end{center}

\noindent
\begin{benchmark} \label{benchmark22} \rm
Evaluate
$$f={}_2F_1\left(-36N,-11N,2N;\frac{1-\frac{1}{N}}{56}\right),\ N \in \mathbb{N}.$$
It stands for the $2 \times 2$ contingency tables
with the row sums $(36N,13N-1)$ and the column sums $(38N-1,11N)$.
The parameter $(p_{ij})$ is set to 
$\left(\begin{array}{cc}
 1 & \frac{1-1/N}{56} \\
 1 & 1 \\
\end{array}
\right)
$.
\end{benchmark}

\begin{benchmark} \label{benchmark35} \rm
Evaluate the expectation for the $3 \times 5$ contingency tables
with the row sums
$(N,2N,12N)$,
the column sums $(N,2N,3N,4N,5N)$,
and the parameter $p$
$$\left(\begin{array}{ccccc}
 1& \frac{ 1}{ 2}& \frac{ 1}{ 3}& \frac{ 1}{ 5}& \frac{ 1}{ 7} \\
 1& \frac{ 1}{ 11}& \frac{ 1}{ 13}& \frac{ 1}{ 17}& \frac{ 1}{ 19} \\
 1&  1&  1&  1&  1 \\
\end{array}\right)$$
\end{benchmark}

\begin{benchmark} \label{benchmark55} \rm
Evaluate the expectation for the $5 \times 5$ contingency tables
with the row sums
$(4N,4N,4N,4N,4N)$,
the column sums $(2N,3N,5N,5N,5N)$,
and the parameter $p$
$$\left(\begin{array}{ccccc}
 1& \frac{ 1}{ 2}& \frac{ 1}{ 3}& \frac{ 1}{ 5}& \frac{ 1}{ 7} \\
 1& \frac{ 1}{ 11}& \frac{ 1}{ 13}& \frac{ 1}{ 17}& \frac{ 1}{ 19} \\
 1& \frac{ 1}{ 23}& \frac{ 1}{ 29}& \frac{ 1}{ 31}& \frac{ 1}{ 37} \\
 1& \frac{1}{37}  &\frac{ 1}{ 41}& \frac{ 1}{ 43}& \frac{ 1}{ 47} \\
 1&  1&  1&  1&  1 \\
\end{array}\right)$$
\end{benchmark}

\begin{benchmark} \label{benchmark77} \rm
Evaluate the expectation for the $7 \times 7$ contingency tables
with the row sums
$(N,2N,3N,4N,5N,6N,7N)$,
the column sums $(N,2N,3N,4N,5N,6N,7N)$,
and the parameter
$$\left(\begin{array}{ccccccc}
 1& \frac{ 1}{ 2}& \frac{ 1}{ 3}& \frac{ 1}{ 5}& \frac{ 1}{ 7}& \frac{ 1}{ 11}& \frac{ 1}{ 13} \\
 1& \frac{ 1}{ 17}& \frac{ 1}{ 19}& \frac{ 1}{ 23}& \frac{ 1}{ 29}& \frac{ 1}{ 31}& \frac{ 1}{ 37} \\
 1& \frac{ 1}{ 41}& \frac{ 1}{ 43}& \frac{ 1}{ 47}& \frac{ 1}{ 53}& \frac{ 1}{ 59}& \frac{ 1}{ 61} \\
 1& \frac{ 1}{ 67}& \frac{ 1}{ 71}& \frac{ 1}{ 73}& \frac{ 1}{ 79}& \frac{ 1}{ 83}& \frac{ 1}{ 89} \\
 1& \frac{ 1}{ 97}& \frac{ 1}{ 101}& \frac{ 1}{ 103}& \frac{ 1}{ 107}& \frac{ 1}{ 109}& \frac{ 1}{ 113} \\
 1& \frac{ 1}{ 127}& \frac{ 1}{ 131}& \frac{ 1}{ 137}& \frac{ 1}{ 139}& \frac{ 1}{ 149}& \frac{ 1}{ 151} \\
 1&  1&  1&  1&  1&  1&  1 \\
\end{array}\right)$$
\end{benchmark} 

\subsection{Floating Point Arithmetic}
If we can evaluate the composition of linear transformations (\ref{GLT})
accurately over floating point numbers,
we can utilize GPU's or other hardware for efficient evaluation.
Unfortunately, we lose the precision during the iteration
of linear transformations in general.
For example,
let us evaluate the case of $N=100$ for our $2 \times 2$ benchmark problem
\ref{benchmark22}
with double arithmetic.
The output by the double precision floating point arithmetic is 
{\tt 4.08315e+94},
but the answer is 
{\tt 4.48194745579962e+94}
where we use the double value expression in the standard form, e.g.,
{\tt 4.08e+94} means $4.08 \times 10^{94}$.
The output by double has only one digit of accuracy.

\subsection{Intermediate Swell of Integers}
We denote by $M(n)$ the complexity of the multiplication of two $n$-digits
integers.
The book \cite{brent-zimmermann-2010} is a survey on algorithms and complexities
on integer arithmetic.

Arithmetic over $\mathbb{Q}$ is more expensive than arithmetic over $\mathbb{Z}$,
because the reduction of a rational number needs the computation of GCD
of the numerator and the denominator.
The best known complexity of the operation of GCD is
$O(M(n) \log n)$ for two $n$-digits numbers
(see, e.g., \cite{muller-2008}, \cite{brent-zimmermann-2010}).
The complexity of the Euclidean algorithm for GCD is $O(n^2)$
\footnote{Timing data over $\mathbb{Q}$ in the version 1 of this paper at arxiv
is very slow, because asir 2000 uses the Euclidean algorithm for the reductions
in $\mathbb{Q}$ as default. The system asir 2018 based on GMP uses faster GCD algorithms 
as default.}.

One way to avoid reductions in $\mathbb{Q}$ in our interations of linear
transformations (\ref{GLT}) is to evaluate numerators and denominators
separately and compute the GCD of the numerator and the denominator
every $R$ step of the linear transformations.
We will call this sequential method 
{\tt g\_mat\_fac\_int} (generalized matrix factorial over integers).
A reduction performing in every $R$ step is necessary.
In fact, our evaluation problems make intermediate swell of integers
by the method {\tt g\_mat\_fac\_int}.
For example, the table below shows sizes of the numerators and the denominators
by the separate evaluation without the intermediate reduction
in our benchmark problem \ref{benchmark22}; \\[0.5ex]
\begin{tabular}{rrrr}
  N    &   digits of num./den.    &       digits of num./den. after reduction & time\\ \hline
 300  &  $1.97\times 10^5/1.96\times 10^5$  &$3.35\times 10^4/3.28\times 10^4$ & 0.92s\\
 500  &  $3.47\times 10^5/3.47\times 10^5$ & $5.87\times 10^4/5.76\times 10^4$ & 1.56s  
\end{tabular} \\[0.5ex]
After the reduction, the numerators and the denominators become smaller
as shown in the second column of the table.

We have no theoretical estimate for the best choice of $R$ for intermediate reductions.
The Figure \ref{fig:graph-test5-interval} is timing data
of our benchmark problem \ref{benchmark35} with $N=100$.
The horizontal axis is the interval $R$ of the intermediate reduction and
the vertical axis is the timing.
\begin{figure}[tb]
\begin{center}
\includegraphics[width=7cm]{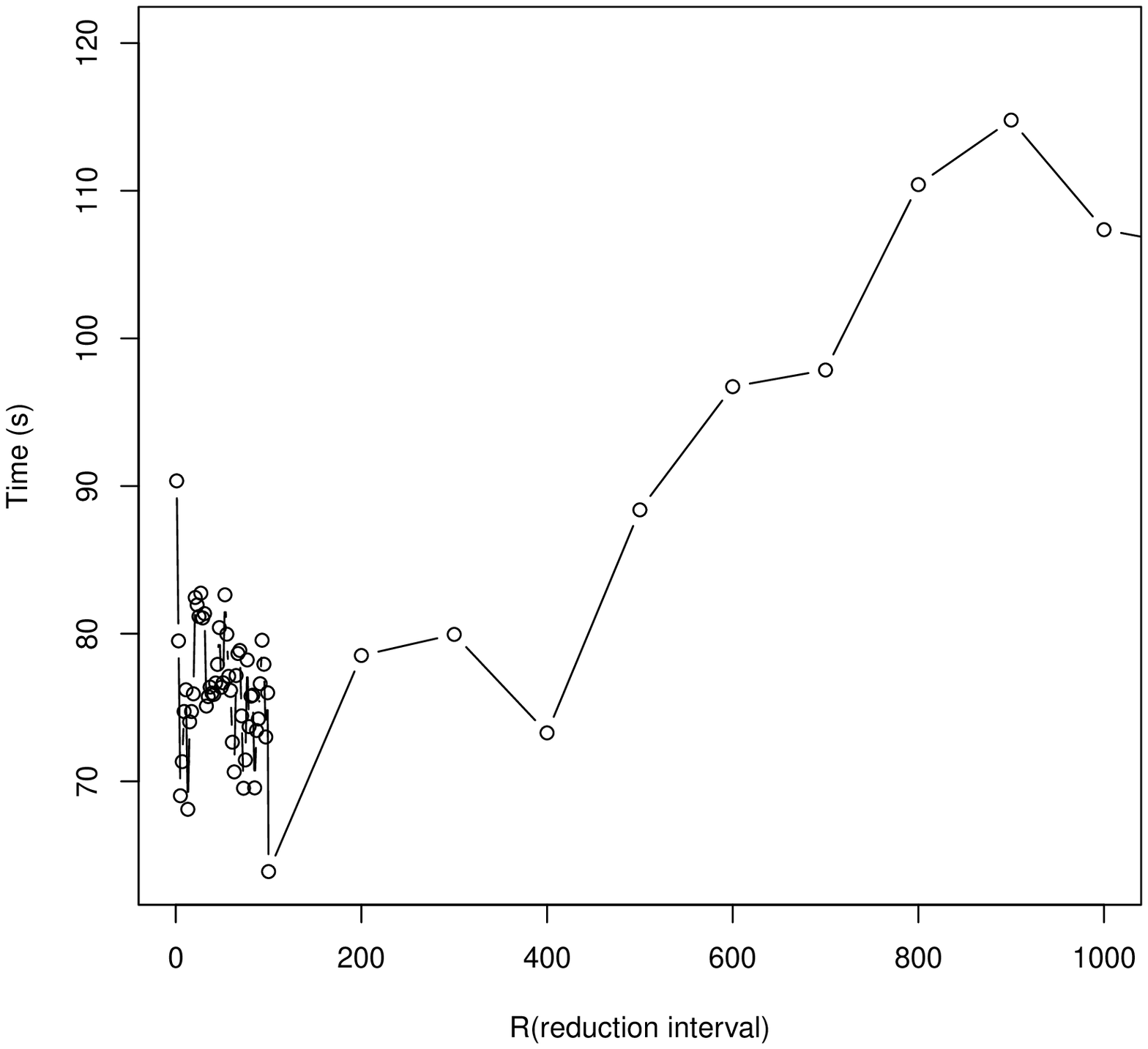}
\end{center}
\begin{center}
\begin{tabular}{r|rrrrrrrr}
$R$ (reduction interval)  & 1 & 3 & 5 & 7 & 9 & 11 & 13 & 15 \\ \hline
Time(s)                   &90.352 & 79.5147 & 69.024 & 71.335 & 74.7312 & 76.2025 & 68.1058 &74.0283 \\
\end{tabular}
\end{center}
\caption{Intermediate reduction} \label{fig:graph-test5-interval}
\end{figure}
The graph indicates that we should choose $R$ such that $ 5 \leq R \leq 100$.

\subsection{Multimodular Method} \label{sec:modular}
It may be standard to use the modular method when we have an intermediate
swell of integers.
We refer to, e.g., \cite{joris2016} and its references 
for the complexity analysis 
on modular methods. 

\begin{algo}[{\tt g\_mat\_fac\_itor} (generalized matrix factorial by itor), 
modular method] \label{algorithm:2}\footnote{
    We use ``itor'' as an abbreviation of the procedure {\tt IntegerToRational}. 
}~

Input: $M(k)$ (matrix), $F$ (vector), $S<E$ (indices), $P_{\rm list}$ (a list of prime numbers), 
$C_{\rm list}$ (a list of processes for a distributed computation).

Output: A candidate value of $M(E)\cdots M(S+2)M(S+1)M(S)F$ or ``failure''.

\begin{enumerate}
 \item Let $F_n$, $F_d$ (scalar), $M_n$, $M_d$ (scalar) be numerators and denominators of $F$ and $M$
       respectively.
 \item For each prime number $P_{i}$ in $P_{\rm list}$, 
       perform the linear transformations
       $ \prod_{i=0}^{E-S}(M_n(S+i) M_d(S+i)^{-1}) F_n F_d^{-1}$
       of $F$ over ${\mathbb F}_{P_{i}}$.
       If the integer $F_d$ or $M_d$ is not invertible modulo $P_{i}$ (unlucky case),
       then skip this prime number $P_{i}$ and set $P_{\rm list}$ to
       $P_{\rm list} \setminus \{ P_i \}$.
       Let the output be $G_{i}$.
       This step may be distributed to processes in the $C_{\rm list}$.
 \item Apply the Chinese remainder theorem to construct a vector $G$ over 
       $\mathbb{Z}/ P\mathbb{Z}$
       satisfying  $G\equiv G_{i} \mod P_{i}$
       where $P=\prod_{P_{i}\in P_{\rm list}}P_{i}$.
 \item Return a candidate value by the procedure {\tt IntegerToRational}$(G,P)$
(rational reconstruction).
\end{enumerate}
\end{algo}

The complexity of the modular method {\tt g\_mat\_fac\_itor}
is estimated as follows.
\begin{thm}\label{itorbc}
Let $n$ be the number of the linear transformations
and the size of the square matrix 
$\displaystyle r=\binom{r_1+r_2-2}{r_1 -1}$.
Suppose that each prime number $P_i$ is $d_p$ digits number
and we use $N_p$ prime numbers.
$C$ is the number of processes. 
The complexity of {\tt g\_mat\_fac\_itor} is approximated as
 $${\rm max}\left\{ O\left(\frac{nr^2N_pM(d_p)}{C}\right),
 O\left(r (d_p N_p)^2\right)\right\}$$
when $n$ is in a bounded region where the rational reconstruction succeeds
and the asymptotic complexity of the Chinese remainder theorem approximates well
the corresponding exact complexity in the region.
\end{thm}

\begin{proof} \rm
We estimate the complexity of each step of {\tt g\_mat\_fac\_itor}.
\begin{enumerate}

\item The complexity of one linear transformation is $O(r^2M(d_p))$.
 The linear transformation is performed $n$ times for $N_p$ prime numbers.
 Then the complexity is $O(nr^2N_pM(d_p))$ on a single process.
 This step can be distributed into $C$ processes, then the complexity
 is  
 $O(\frac{nr^2N_pM(d_p)}{C})$.

\item   The complexity  to find an integer $x$ such that 
      $x \equiv x_i \mod p_i\ (i=1, \ldots, N_p)$
 is discussed in \cite[Theorem 6]{joris2016} under the assumption that an inborn FFT scheme is used.
It follows from the estimate that the reconstruction complexity $C_n(N_p)$ 
of $N_p$ primes of $d_p$ digits is bounded by
$ (2/3 + o(1)) M(d_p N_p) \, {\rm max}\, 
\left( \frac{\log N_p}{\log \log (d_p N_p)}, 1+ O(N_p^{-1}) \right)
$

\item The rational reconstruction algorithm {\tt IntegerToRational},
see, e.g., \cite{gathen-gerhard}, \cite{itor}, is a variation of the Euclidean 
  algorithm and its complexity is bounded by $O((N_p d_p)^2)$.
We have $r$ numbers to reconstruct.
\end{enumerate}
Since the complexity of the step 2 is smaller than other parts, we obtain 
the conclusion.
\end{proof}

\begin{figure}[tb]
\begin{center}
\includegraphics[width=7cm]{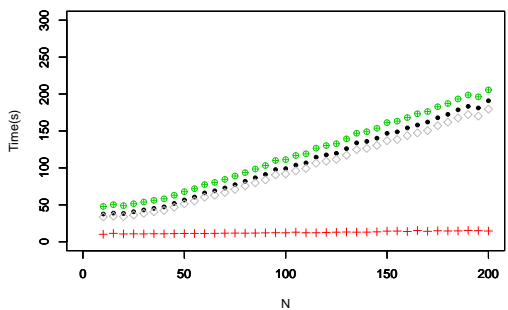}
\end{center}
\caption{$5 \times 5$ contingency table, the benchmark problem \ref{benchmark55}
with $32$ processes}
\label{fig:graph-time2-crt-5-32}
\end{figure}

The complexity is linear with respect to $n$ (which is proportional to the size of the marginal sum vector in our benchmark problems)
when the first argument of the ``max'' in the theorem is dominant.
However, when $n$ becomes larger,
the rational reconstruction fails or gives a wrong answer.
This is the reason why we give the assumption that $n$ is in a bounded region.
Note that the complexity estimate in the theorem is not an asymptotic complexity
and is an approximate evaluation of it.

Let us present an example that this approximate evaluation works.
Figure \ref{fig:graph-time2-crt-5-32} is a graph of the timing data
for the benchmark problem \ref{benchmark55} with $N_p = 400$
and $d_p=100$ by the decimal digits.
The top point graph is the total time, the second top point graph is
the time of the generalized matrix factorial
(the execution time of Algorithm \ref{algorithm:2}), 
the third point graph is the time
of the distributed generalized matrix factorial by modulo $P_i$'s
(the step 2 of Algorithm \ref{algorithm:2}).
The last point graph is the time to obtain contiguity relations.
Contiguity relations for several directions are obtained by distributing
the procedures into $32$ processes.
Note that the point graph is linear with respect to $N$,
which is proportional to the number of the linear transformations $n$.
The timing data imply that the first argument of ``max'' of Theorem \ref{itorbc}
is dominant in this case.
In fact, when $N=200$, the step for reconstructing rational numbers
only takes about 8 seconds and linear transformations over finite fields
take from 35 seconds to 52 seconds.

We should ask if our multimodular method is efficient on real computer
environments.
The following table is a comparison of timing data
of the sequential method {\tt g\_mat\_fac\_int} (with a distributed computation
of contiguity relations by $32$ processors)
and the multimodular method 
{\tt g\_mat\_fac\_itor} by $32$ processors
for the benchmark problem \ref{benchmark55}. \\[0.5ex]
\begin{tabular}{l|rr}
N                                             &  90   &  200    \\ \hline
{\tt g\_mat\_fac\_int} with the reduction interval $R=100$  & 21.57 &  45.40  \\
{\tt g\_mat\_fac\_int} without the intermediate reduction & 68.17 &  227.23  \\
{\tt g\_mat\_fac\_itor} by $32$ processors    & 103.23 & 205.57 \\
\end{tabular} \\[0.5ex]
Unfortunately, the multimodular method is slower than the sequential method
{\tt g\_mat\_fac\_int} with a relevant choice of $R$
on our best computer, 
however it is faster than the case of a bad choice of $R=\infty$.

\begin{figure}[tb]
\begin{center}
\includegraphics[width=7cm]{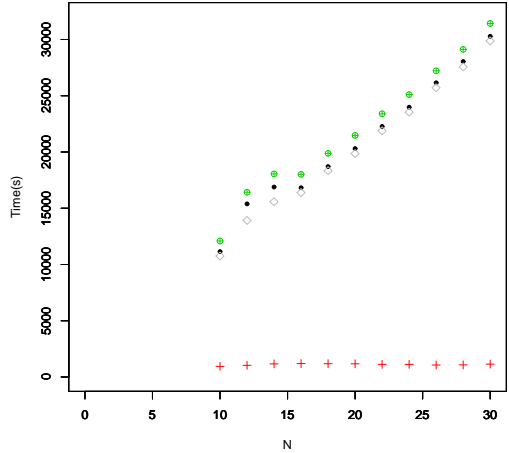}
\end{center}
\caption{$7 \times 7$ contingency table, the benchmark problem \ref{benchmark77}
with $32$ processes}
\label{fig:graph-time2-crt-7-32}
\end{figure}

When the size of contingency table becomes larger, 
the rank $r$ becomes larger rapidly.
For example, $r=20$ for the $5 \times 5$ contingency tables
and $r=924$ for the $7 \times 7$ contingency tables.
The Figure \ref{fig:graph-time2-crt-7-32} shows timing data of our benchmark
problem \ref{benchmark77} of $7 \times 7$ contingency tables
with the multimodular method by $32$ processors.
We can also see linear timing with respect to $N$,
but the slope is much larger than the $5 \times 5$ case
as shown in our complexity analysis.

\subsection{Binary Splitting Method}
It is well-known that
the binary splitting method for the evaluation of the factorial $m!$
of a natural number $m$
is faster method than a naive evaluation of the factorial
by $m! = m \times (m-1)!$.
The binary splitting method evaluates
$ m (m-1) \cdots (\lfloor m/2 \rfloor+1)$ and $\lfloor m/2 \rfloor (\lfloor m/2 \rfloor-1) \cdots 1$
and obtains $m!$.
This procedure can be recursively executed.
This binary splitting can be easily generalized to our generalized
matrix factorial; 
we may evaluate, for example, 
$M(a) M(a+1) \cdots M(\lfloor a/2 \rfloor-1)$
and
$M(\lfloor a/2 \rfloor) \cdots M(-2)$
to obtain  $M(a) M(a+1) \cdots M(-2)$, $a < -2$
in (\ref{eq:composite_of_linear_transformation}).
This procedure can be recursively applied.
However, what we want to evaluate is the application of the matrix
to the vector $F(-1)$.
The matrix multiplication is slower than the linear transformation.
Then, we cannot expect that this method is efficient for our problem
when the size of the matrix is not small and the length of multiplication
is not very long.
However, there are cases
that the binary splitting method is faster.
Here is an output by our package {\tt gtt\_ekn3.rr}.
{\footnotesize
\begin{verbatim}
[1828] import("gtt_ekn3.rr")$
[4014] cputime(1)$
0sec(1.001e-05sec)
[4015] gtt_ekn3.expectation(Marginal=[[1950,2550,5295],[1350,1785,6660]],
                    P=[[17/100,1,10],[7/50,1,33/10],[1,1,1]]|bs=1)$ //binary splitting
3.192sec(3.19sec)
[4016] gtt_ekn3.expectation(Marginal,P)$
4.156sec(4.157sec)
\end{verbatim}
}

\subsection{Benchmark of Constructing Contiguity Relations}
We gave a complexity analysis of finding contiguity relations.
When $r_1$ is fixed, it is $O(r_2^{3r_1})$.
\begin{figure}[tb]
\begin{center}
\includegraphics[width=7cm]{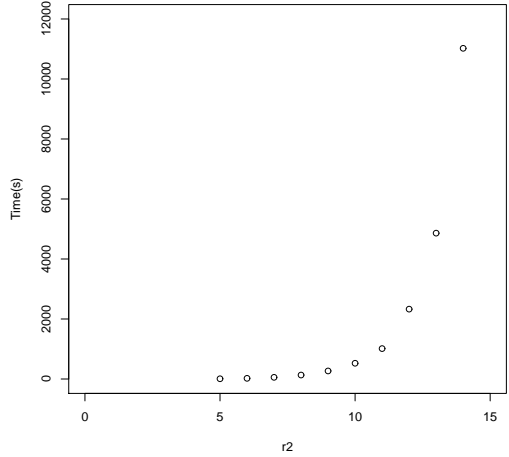}
\end{center}
\begin{center}
\end{center}
\caption{Time to obtain contiguity relations} \label{fig:graph-time-upAlpha}
\end{figure}
The Figure \ref{fig:graph-time-upAlpha} shows timing data to obtain 
contiguity relations for $5 \times r_2$ contingency tables
where the parameter $p$ is 
$\left(\begin{array}{ccccc}
1 & 1     & 1    & \cdots  & 1 \\
1 & 1/p_1 & 1/p_2 & \cdots  & 1/p_{r_2-1} \\
1 & 1/p_{r_2} & 1/p_{r_2+1} & \cdots & 1/p_{2(r_2-1)} \\
1 & \cdots \\
1 & 1/p_{(r_1-1)(r_2-1)+1} & \cdots \\
\end{array}\right)
$
($p_i$ is the $i$-th prime number),
the row sum vector is $(a_1,400,400,400,400)$,
and
the column sum vector is $(200,300,500,500, \ldots, 500)$.
As is shown by our complexity analysis, when $r_2$ becomes larger,
it rapidly becomes harder to obtain contiguity relations.

\section{Zero Cells} \label{sec:zero}

The contiguity relations derived by \cite{MG} are valid
only when there are no zero cells in the contingency table.
If there is a zero ($p_{ij}=0$ and $u_{ij}=0$) in the contingency table, 
a denominator of the contiguity relation is zero in general
and therefore we cannot use their identity.
One method to avoid this difficulty is interpolation.
Note that the normalizing constant $Z$ is a rational function
in $p_{ij}$ and the expectation 
$E[U_{ij}] = p_{ij}\frac{\partial \log Z}{\partial p_{ij}}$
is also a rational function.
Because it is a rational function, we can obtain the exact value
by evaluating it on a sufficient number of rational $p_{ij}$'s.
\begin{prop}
Let $\beta$ be the marginal sum vector and $L$ a generic line in $p$-space.
If we evaluate $E[U_{ij}]$ at $2 \beta_1$  points 
$p \in {\mathbb R}_{>0}^{r_1 \times r_2}$ on a line $L$,
then the exact value of $E[U_{ij}]$ can be obtained 
at any point on $L$.
\end{prop}

\begin{proof} \rm
When we restrict $E[U_{ij}]$ to the line $L$,
it is a rational function in one variable.
The degree of the denominator and the numerator is $\beta_1$ at most.
Apply an interpolation algorithm by rational function,
e.g., Stoer-Bulirsch algorithm \cite{SB}, \cite{recipes}.
Then, we can obtain the exact value by interpolation.
\qed
\end{proof}

\begin{example}  \rm
Let the marginal sums and the parameter $p$ (cell probability) be
$$
\begin{array}{ccc|c}
* & * & * & 3 \\
* & * & * & 4 \\
* & * & * & 3 \\ \hline
3 & 4 & 3 &
\end{array},
p = \left(
\begin{array}{ccc}
1 & 1/2 & 0 \\
1 & 1/3 & 1/4 \\
1 &1 &1 \\
\end{array}
\right)
$$
Then, we can evaluate the expectation matrix $(E[U_{ij}])$
by the difference HGM and interpolation.
Below is an output of our package {\tt gtt\_ekn3}.
Here the {\tt randinit} parameter specifies an interval of random non-zero
$p_{ij}$'s where $(i,j)$'s are positions of zero cells.
{\footnotesize
\begin{verbatim}
[5150] import("gtt_ekn3.rr");
0
[5151] E=gtt_ekn3.cBasistoE_0(0,[[3,4,3],[3,4,3]],[[1,1/2,0],[1,1/3,1/4],[1,1,1]] | randinit=20);
[ 71076/56575 98649/56575 0 ]
[ 157581/113150 28069/22630 77337/56575 ]
[ 39717/113150 114957/113150 92388/56575 ] 
// Expectation (exact value)
[5153] number_eval(E);  // Expectation (approximate value)
[ 1.25631462660186 1.74368537339814 0 ]
[ 1.39267344233319 1.2403446752099 1.36698188245692 ]
[ 0.351011931064958 1.01596995139196 1.63301811754308 ]
\end{verbatim}
}
\end{example}

Although the interpolation method is applicable to any pattern of $0$-cells,
a more efficient method involves utilizing 
hypergeometric functions restricted on some $p_{ij}=0$'s. 
In general, contiguity relations and Pfaffian systems for such hypergeometric functions become complicated. 
In \cite{goto-0}, a method is put forward to evaluate intersection numbers 
and contiguity relations when only one $p_{ij}$ is zero.

\section{Sufficient Statistics as $\sigma$-algebra } \label{sec:sufficient}
It is often that we decompose parameters for contingency tables
into row and column probabilities and odds ratios.
When only odds ratios are the parameters of interest,
CMLE is an appropriate method to estimate those odds ratios.
However, this decomposition is no longer elementary when contingency tables
contain zero cells.
To facilitate a mathematically clear discussion of CMLE in the next section,
we offer a new formulation of parameters of interest,
nuisance parameters, and sufficient statistics.

Classical formulations of sufficient statistics as $\sigma$-algebras
appear in, e.g., \cite{billingsley1995probability}, \cite{Landers-Rogge1972}.
Our formulation is different because we treat parameters as random variables
instead of considering a family of probability measures.
This Bayesian statistical approach enable us to consider
$\sigma$-algebras on parameter spaces.
We express nuisance parameters and parameters of interest as
sub $\sigma$-algebras of the $\sigma$-algebra generated by all parameters and data.
A Bayesian approach to sufficient statistics is presented in, e.g.,
Chapter 2 of the text book by M.Schervish \cite{Schervish}.
This text book studies sufficient statistics by conditional probabilities
given parameter valued random variables.
We study them by a more general approach of conditional expectations 
given $\sigma$-algebras.
The technical details are lengthy to be precise and, in this section and the next section,
we state only fundamental notions and theorems which we need to study
two way contingency tables. 
Proofs for them are given in the preprint of this paper at arxiv
\footnote{\url{https://arxiv.org/abs/1803.04170}}.
A general framework of the theory will be given in 
\cite{koyama-sufficient}.

The treatment of nuisance parameters and parameters of interest
is an important issue in statistics.
The distinction between those parameters which are of interest
versus those which are nuisance, may seem easy.
In fact, it seems to be only a matter of declaring
that $\mu$ is a parameter of interest
or $\nu$ is a nuisance parameter.
As we will see in the next section,
when a group acts on parameter spaces
and the group is regarded as
the space of nuisance parameters, the distinction
between them is not trivial.
From a geometric perspective,
the cause of this difficulty is that
determining whether a parameter is ``of interest'' or a ``nuisance'' depends on a coordinate system.
To formulate the ``of interest'' notion
independently of a specific coordinate system,
we will consider $\sigma$-algebras on parameter spaces.
In probability theory and stochastic processes,
$\sigma$-algebra is important as a natural way to express information
(see, e.g. \cite{karatzas1988brownian}).
Discussions in this section are based on
conditional expectations with respect to $\sigma$-algebra.
For basic properties of conditional expectation,
see \cite{williams1991probability}.

Let $\Theta$ be a set.
The set $\Theta$ stands for the parameter spaces.
Let $\mathcal B(\Theta)$ be a $\sigma$-algebra on $\Theta$,
then $(\Theta, \mathcal B(\Theta))$ is a measure space.
In the case where $\Theta$ is a topological space,
we assume that $\mathcal B(\Theta)$ is the Borel algebra on
$\Theta$.

In standard parameter estimation,
we assume a probability space $(\Omega',\mathcal F',\Prob_c')$
with a parameter $c \in \Theta$.
Let us define our probability space from the standard setting.
Suppose $(\Theta,\mathcal B(\Theta),\mu)$
is a probability space.
Put $\Omega:=\Omega'\times\Theta$.
Let $\mathcal F$ be the $\sigma$-algebra on $\Omega$
generated by
$$
A\times B:=
\{(\omega,c)\in\Omega\vert \omega\in A,\, c\in B\}
\quad (A\in\mathcal F',\, B\in\mathcal B(\Theta)).
$$
The measurable space  $(\Omega,\mathcal F)$ is deemed to be
the product measurable space
of $(\Omega',\mathcal F')$ and 
$(\Theta,\mathcal B(\Theta))$
\cite[p75,Lemma]{williams1991probability}.
For $A\in\mathcal F'$, let $f_A:\Theta\rightarrow\mathbf R$
be the function defined by
$
f_A(c):= \int_A \Prob_c'(d\omega)
\,(c\in\Theta).
$
If $f_A$ is $\mathcal B(\Theta)$-measurable
for any $A\in\mathcal F'$,
we can define a measure $\Prob$ on $\mathcal F$ by
$
\Prob(A\times B):=\int_B f_A(c)\mu(dc)
\,(A\in\mathcal F',\,B\in\mathcal B(\Theta))
$.
Thus, our probability space is defined as the product space
under the measurable condition of $f_A$.

Let $\theta$ be a measurable map from $\Omega$ to $\Theta$
defined by
$$ \theta: \Omega \ni (\omega',c) \mapsto c \in \Theta. $$
This implies that parameters can be regarded as a $\Theta$-valued random variable.
Although random variables are usually denoted by capital letters,
we use lower case letters to denote random variables that are regarded as parameters.
\begin{example}\rm
  Let $(\Omega',\mathcal F',\Prob_c')$ be
  the probability space $(\mathbf R, \mathcal B(\mathbf R), N(\mu,\sigma^2))$,
  where $N(\mu,\sigma^2)$ is the Gaussian distribution on $\mathbf R$
  with mean $\mu$ and variance $\sigma^2$.
  In this case, the parameter space is
  $\Theta=\{(\mu,\sigma^2)\in\mathbf R^2\vert \sigma^2>0 \}$
  and the parameter $\theta$ as a measurable map is defined by
  $$
  \theta: \Omega\ni (x,(\mu,\sigma^2))\mapsto (\mu,\sigma^2)\in\Theta.
  $$
\end{example}

We restart from 
a probability space $(\Omega,\mathcal F,\Prob)$,
which is not necessarily a product space.
For a sub $\sigma$-algebra $\mathcal G$ of $\mathcal F$,
we use $\mathcal L^1(\mathcal G)$ to denote
the linear space of random variables
which are integrable and $\mathcal G$-measurable.
When two elements $X$ and $Y$ of $\mathcal L^1(\mathcal G)$
satisfy $X(\omega)=Y(\omega)$ for all $\omega\in\Omega$,
we say that $X$ and $Y$ are equal and denote $X=Y$.
Note that $X=Y$ almost surely does not imply that $X=Y$.
Let $\vartheta$ be the sub $\sigma$-algebra of $\mathcal F$ generated 
by a random variable $\theta$.
It represents the information of $\theta$.
We formulate notions of nuisance parameters, sufficient parameters, 
and parameters of interest
as sub $\sigma$-algebras of $\vartheta$.

For a  pair of random variables $X$ and $Y$,
$Y$ is $\sigma(X)$-measurable if and only if
$Y$  equals to $f(X)$ for a Borel measurable function $f$.
See, e.g., \cite[p206]{williams1991probability}.

Let $X$ and $Y$ be ${\mathbf R}$-valued random variables
and $\theta$ be a $\Theta$-valued random variable,
which we will call a parameter.
We assume that $X$ is integrable.
The conditional expectation $\E{X\vert Y,\theta}$
can be regarded as a function of $(Y,\theta)$, i.e., 
we can take a Borel measurable function $f$
from $\mathbf R\times\Theta$ to $\mathbf R$
such that 
$$
f(Y,\theta) = \E{X\vert Y,\theta}
\quad a.s.
$$

Because the equation $f(y,c_1)=f(y,c_2)$ may hold
even if $c_1\neq c_2$,
the conditional expectation $\E{X\vert Y,\theta}$ can be measurable
with respect to a sub $\sigma$-algebra strictly smaller than $\sigma(Y,\theta)$.
This suggests that taking conditional expectation can reduce
the information of $\theta$.

Let us express this loss of information of $\theta$
in terms of $\sigma$-algebra.
Let $\ObservedData$ and $\mathcal G$ be a sub $\sigma$-algebras of $\mathcal F$.
In some applications,
such as 
Theorem \ref{b25Dec2017} discussed later, 
it is assumed that
$\ObservedData$ is the sub $\sigma$-algebra
generated by all observable statistics
and
$\mathcal G$ is a sub $\sigma$-algebra generated
by a fraction of the observable statistics and a fraction of the parameters.
Note that $\mathcal G$ may include some information of parameters.
For $X\in\mathcal L^1(\ObservedData)$, the conditional expectation
$\E{X\vert\mathcal G}$ can be measurable for a sub $\sigma$-algebra
which is strictly smaller than $\mathcal G$.

\begin{definition}\label{b19Dec2017} \rm
  Sub $\sigma$-algebra $\mathcal I$ is said to be {\it of interest}
  with respect to a pair of sub $\sigma$-algebras $(\ObservedData, \mathcal G)$
  if, for all $X\in\mathcal L^1(\ObservedData)$, a version of
  $\E{X\vert\mathcal G}$ is $\mathcal I$-measurable.
\end{definition}

Notions of nuisance and sufficiency describe
a special case of such information loss.
\begin{definition}\label{a14Nov2017}
  \rm
  Let $\ObservedData$, $\mathcal S$ and $\mathcal N$ be
  sub $\sigma$-algebras  of $\mathcal F$.
  When $\mathcal S$ is of interest with respect to
  $(\ObservedData, \sigma(\mathcal S,\mathcal N))$,
  we deem that $\mathcal S$ is sufficient for $(\ObservedData,\mathcal N)$
  or that $\mathcal N$ is nuisance for $(\ObservedData,\mathcal S)$.
\end{definition}

\begin{remark} \rm
  Note that the condition of Definition \ref{a14Nov2017}
  is equivalent to stating that the equation
  \begin{equation}\label{20171004b}
    \E{X\vert\sigma(\mathcal S,\mathcal N)}
    =\E{X\vert\mathcal S}
    \quad\almostsurely
  \end{equation}
  holds for any $X\in\mathcal L^1(\ObservedData)$.
  In fact, we have
  \begin{align*}
    \E{X\vert\sigma(\mathcal S,\mathcal N)}
    &=
    \E{\E{X\vert\sigma(\mathcal S,\mathcal N)}\vert \mathcal S}
    &&\(\because \E{X\vert\sigma(\mathcal S,\mathcal N)}\in\mathcal L^1(\mathcal S)\)
    \\&=
    \E{X\vert\mathcal S}
    &&\(\because \text{tower property}\).
  \end{align*}
\end{remark}

\begin{remark}\label{a22Nov2017} \rm
In statistics,
a statistic $T$ is sufficient with respect to a parameter $\theta$
if the conditional distribution of observed data $X$
given the statistic $T=t$ does not depend on the parameter $\theta$.
This condition is formally expressed as 
$$
p(x\vert t,\theta) = p(x\vert t).
$$
In similar tests and the Neyman--Scott Problem,
$\theta$ is denoted as a nuisance parameter or an uninteresting parameter
\cite{amari2016information}.
We express this condition in terms of the measure theory
in Definition \ref{a14Nov2017}.
In our definition, we use $\sigma$-fields instead of
statistics and parameters.
Traditional definitions can be reduced
to our definition by
\begin{align*}
\ObservedData &=\sigma\(X\), &
\mathcal S &= \sigma\(T\), &
\mathcal N &=\sigma\(\theta\).
\end{align*}
Intuitively, 
$\ObservedData$, $\mathcal S$, and $\mathcal N$
denote the information of
the observed data, the sufficient statistics, and the nuisance parameters,
respectively.

In addition, we utilize conditional expectations
instead of conditional probabilities
because the latter
can only be defined for a limited class of probability space and conditions.

Fundamental theorems on sufficient statistics can be generalized in our
formulation on the sufficient sigma field \cite{koyama-sufficient}.
\end{remark}

\begin{example}\label{a25Dec2017} \rm
  For random variables  $X_1,\dots,X_n,\theta$,
  suppose that 
  \begin{enumerate}
  \item $0\leq\theta\leq 1$
  \item The conditional probability of $X_1,\dots, X_n$
    for given $\theta$ is 
    $$
    \Prob\left(X_1=x_1,\dots, X_n=x_n\vert \theta\right) =
    \prod_{i=1}^n \theta^{x_i}\(1-\theta\)^{1-x_i}
    \quad (x_i\in\{0,1\})
    $$
  \end{enumerate}
  Then, putting
  $\ObservedData:= \sigma(X_1,\dots,X_n)$,
  $\mathcal N := \sigma(\theta)$,
  $\mathcal S:= \sigma(X_1+\cdots+X_n)$,
Hence,  $\mathcal S$ is sufficient for $(\ObservedData,\mathcal N)$.

In order to clarify our formulation by the $\sigma$-algebra,
we will prove that ${\cal S}$ is sufficient.
For $x=(x_1,\dots,x_n)^\top\in\mathbf R^n$,
we denote by $|x|$ the sum of elements of $x$.
Put $X:=(X_1,\dots, X_n)^\top$ and $T:=|X|=X_1+\cdots+X_n$.
For any $Y\in\mathcal L^1(\mathcal D)$,
we can take a Borel measurable function $f:\mathbf R^d\rightarrow\mathbf R$
such that $Y=f(X)$.
Let $g:\{0,1,\dots,n\}\rightarrow\mathbf R$ be a function defined by
$$
g(t):=   \binom{n}{t}^{-1}\sum_{x\in\{0,1\}^n}\delta_{t,|x|} f(x).
$$
Then, $g(T)$ is $\mathcal S$-measurable.
For any $B, C\in\mathcal B(\mathbf R)$, we have
\begin{align*}
  &
  \E{Y ; T\in B, \theta\in C}
  \\&
  =  \E{YI_B(T)I_C(\theta)}
  \\&
  =  \E{f(X)I_B(|X|)I_C(\theta)}
  \\&
  =\int \sum_{x\in\{0,1\}^n}
  f(x) I_B(|x|)I_C(\theta)
  \prod_{i=1}^n\theta^{x_i}(1-\theta)^{1-x_i}p(\theta)
  d\theta
  \\&
  =\int \sum_{x\in\{0,1\}^n}
  f(x) I_B(|x|)I_C(\theta)
  \prod_{i=1}^n\theta^{x_i}(1-\theta)^{1-x_i}p(\theta)
  d\theta
  \\&
  =\int \sum_{x\in\{0,1\}^n}\sum_{t=0}^n \delta_{t,|x|}
  f(x) I_B(|x|)I_C(\theta)
  \prod_{i=1}^n\theta^{x_i}(1-\theta)^{1-x_i}p(\theta)
  d\theta
  \\&
  =\int \sum_{x\in\{0,1\}^n}\sum_{t=0}^n \delta_{t,|x|}
  f(x) I_B(t)I_C(\theta)
  \theta^t(1-\theta)^{n-t}p(\theta)
  d\theta
  \\&
  =\int \sum_{t=0}^n
  \binom{n}{t}^{-1}\sum_{x\in\{0,1\}^n}\delta_{t,|x|} f(x)
  I_B(t)I_C(\theta)
  \binom{n}{t}\theta^t(1-\theta)^{n-t}p(\theta)  d\theta
  \\&
  =\E{ g(T)I_B(t)I_C(\theta)  }
  \\&
  =\E{g(T) ; T\in B, \theta\in C }.
\end{align*}
Here, we denote by $I_B$ and $I_C$
the indicator functions of $B$ and $C$ respectively.
Since $\sigma(\mathcal S,\mathcal N)$ is generated by
$\{T\in B\}\cap\{\theta\in C\}\,(B,C\in\mathcal B(\mathbf R))$,
by \cite[1.6.~Lemma (a)]{williams1991probability},
we have $\E(Y;A)=\E(g(T);A)$
for any $A\in\sigma(\mathcal S,\mathcal N)$.
Consequently, $g(T)$ is a version of $\E(Y\vert\sigma(\mathcal S,\mathcal N))$.
Hence, $\mathcal S$ is sufficient for $(\mathcal D,\mathcal N)$.
\end{example}

To describe a sub $\sigma$-algebra of interest
in our application to the $\mathcal A$-distribution,
we consider orbits of some group action.
Suppose that a group $G$ acts on a measurable space $(S,\Sigma)$.
For $B\subset S$ and $g\in G$, we put
\begin{align*}
  g\cdot B&:=\left\{g\cdot b \,\vert\, \,b\in B \right\}, &
  G\cdot B&:=\left\{g\cdot b \,\vert\, g\in G,\,b\in B \right\}.
\end{align*}
Note that $G\cdot B=B$ holds if and only if $g\cdot B=B$ for any $g\in G$.

Let $\mathcal O^*$ be the family of 
the element in $\Sigma$ invariant under the action of $G$,
i.e., we put
$$
\mathcal O^* :=\left\{  B\in\Sigma : G\cdot B=B \right\}.
$$
\begin{lemma}\label{a18Dec2017}
  $\mathcal O^*$ is a sub $\sigma$-algebra of $\Sigma$. 
\end{lemma}
\begin{proof} \rm
  Obviously, $\mathcal O^*$ includes $\Sigma$.
  Let $B$ be an element in $\mathcal O^*$.
  Take any $g\in G$ and $b'\in B^c$. Suppose that $g\cdot b'\in B$.
  Then, we have $b'=g^{-1}gb'\in G\cdot B=B$. This is a contradiction.
  Hence, $g\cdot b'$ is an element of $B^c$, and we have $G\cdot B^c\subset B^c$.
  Since $G\cdot B^c$ includes $B^c$ obviously, we have $G\cdot B^c= B^c$.
  Consequently, $\mathcal O^*$ includes $B^c$.

  Suppose that $B_n\,(n\in\mathbf N)$ is an element of $\mathcal O^*$.
  Since we have
  $
  G\cdot\bigcup_{n=1}^\infty B_n
  =\bigcup_{n=1}^\infty G\cdot B_n
  =\bigcup_{n=1}^\infty B_n,
  $
  $\bigcup_{n=1}^\infty B_n$ is an element of $\mathcal O^*$.
  \qed
\end{proof}
A measurable map  $X:(\Omega,\mathcal F)\rightarrow(S,\Sigma)$
induces a sub $\sigma$-algebra of $\mathcal F$ by
$$
\mathcal O := \left\{
\{ X\in B\} \vert B\in\mathcal O^*
\right\}.
$$
Note that $\{ X\in B \}$ is the inverse image
$X^{-1}(B)=\{\omega\in\Omega\vert X(\omega)\in B\}$.
This notation is often used in the probability theory
and we use it in the sequel.
We call $\mathcal O$ as the $\sigma$-algebra generated by the orbits of group $G$.
\begin{lemma}\label{a19Dec2017}
  Let $f:S\rightarrow\mathbf R$ be a function.
  Suppose that a measurable map
  $X:(\Omega,\mathcal F)\rightarrow (S,\Sigma)$
  is surjective.
  Then, all of the following four conditions are equivalent:
  \begin{enumerate}
  \item[\rm (a).] $f$ is $\mathcal O^*$-measurable.
  \item[\rm (b).] $f(g\cdot x)=f(x)$ holds for any $g\in G$ and any $x\in S$.
  \item[\rm (c).] $f(X)$ is $\mathcal O$-measurable.
  \item[\rm (d).] $f(g\cdot X)=f(X)$ holds for any $g\in G$.
  \end{enumerate}
\end{lemma}
\begin{proof} \rm
  [(a)$\Rightarrow$(c)]\/.
  Suppose that $f$ is $\mathcal O^*$-measurable.
  For any $B\in\mathcal B(\mathbf R)$, we have $f^{-1}(B)\in\mathcal O^*$.
  By the definition of $\mathcal O$,
  $X^{-1}(f^{-1}(B))=\{f(X)\in B\}\in\mathcal O$ holds.
  Hence, $f(X)$ is $\mathcal O$-measurable.

  [(c)$\Rightarrow$(d)]\/.  
  Suppose that $f(X)$ is $\mathcal O$-measurable.
  Take an arbitrary $a\in\mathbf R$.
  Then, $\{f(X)=a\}\in\mathcal O$ implies
  that there exists $B\in\mathcal O^*$
  such that $\{X\in B\}=\{f(X)=a\}=\{X\in f^{-1}(a)\}$.
  Since $X:\Omega\rightarrow S$ is surjective,
  we have $B=f^{-1}(a)$.
  Thus, $f^{-1}(a)$ is an element of $\mathcal O^*$,
  and we have $G\cdot f^{-1}(a)=f^{-1}(a)$.
  This implies $g\cdot f^{-1}(a)=f^{-1}(a)$ holds for any $g\in G$, 
  and we have
  \begin{align*}
    \{f(g\cdot X)=a\}
    &=\{g\cdot X\in f^{-1}(a)\}
    =\{ X\in g^{-1}\cdot f^{-1}(a)\}
    \\&
    =\{ X\in f^{-1}(a)\}
    =\{f( X)=a\}.
  \end{align*}
  Hence, $f(g\cdot X)=f(X)$ holds for all $g\in G$.

  [(d)$\Leftrightarrow$(b)]\/.  
  Since $X:\Omega\rightarrow S$ is surjective,
  $f(g\cdot x)=f(x)$ holds for any $x\in S$ and $g\in G$
  if and only if 
  $f(g\cdot X)=f(X)$ holds for any $g\in G$.

  [(b)$\Rightarrow$(a)]\/.  
  Suppose the function $f$ is invariant under the action of $G$.
  Take an arbitrary $B\in\mathcal B(\mathbf R)$.
  By the invariance of the function $f$, we have
  \begin{align*}
    G\cdot f^{-1}(B)
    &
    =\left\{g\cdot x \vert g\in G,\, x\in S,\, f(x)\in B\right\}
    =\left\{x\in S \vert f(x)\in B\right\}
    =f^{-1}(B).
  \end{align*}
  Hence, $f^{-1}(B)$ is included in $\mathcal O^*$.
  This implies that $f$ is $\mathcal O^*$-measurable.
  \qed
\end{proof}
\begin{remark}\label{rem-bac}
  In Lemma \ref{a19Dec2017},
  $\mathrm{(b)}\Rightarrow\mathrm{(a)} \Rightarrow\mathrm{(c)}$
  holds even without the assumption that $X$ is surjective.
\end{remark}
Since $\mathcal O$ is a sub $\sigma$-algebra of $\sigma(X)$,
$Y\in\mathcal L^1(\mathcal O)$ can be regarded as a function of $X$.
By Lemma \ref{a19Dec2017},
we say that a random variable $Y$ is invariant under the action of group $G$
if $Y$ is $\mathcal O$-measurable.

We apply the above discussion on group actions to sufficient $\sigma$-algebras.
Let $\ObservedData$ be a sub $\sigma$-algebra of $\mathcal F$.
Let $\Theta$ be a topological space,
and $\theta:(\Omega,\mathcal F)\rightarrow(\Theta,\mathcal B(\Theta))$
be a measurable map.
We regard $\theta$ and $\Theta$
as the parameter and the space of parameters respectively.
Let $S$ be a measurable space,
and $T:\Omega\rightarrow S$ be an $\ObservedData$-measurable map.
For $X\in\mathcal L^1(\ObservedData)$,
$\E{X\vert T,\theta}$ can be regarded
as a function on $S\times\Theta$.
In other words, there exists a function $f_X:S\times\Theta\rightarrow\mathbf R$
such that
$
f_X(T(\omega),\theta(\omega)) = \E{X\vert T,\theta}(\omega)
$
for all $\omega\in\Omega$.
\begin{lemma}\label{a22Dec2017}
  We assume the same notation as above.
  Suppose that an action of group $G$ on $\Theta$ satisfies 
  $$
  f_X(t,g\cdot c) = f_X(t, c)
  $$
  for all $t\in S$, $c\in\Theta$, $g\in G$, and $X\in\mathcal L^1(\ObservedData)$,
  and put
  $$
  \mathcal O
  :=\{
   \{(T,\theta)\in B\} \,\vert\,
   B\in \Sigma\times\mathcal B(\Theta),\,
   G\cdot B=B
  \}.
  $$
  Then, $\mathcal O$ is of interest with respect to
  $(\ObservedData, \sigma(T,\theta))$.
\end{lemma}
\begin{proof}\rm
  The group action on $\Theta$ induces an group action on
  the Cartesian product $S\times\Theta$ by
  $$
  g\cdot(t,c) = (t,g\cdot c)
  \quad (g\in G,\, (t,c)\in S\times\Theta).
  $$
  Applying Lemma \ref{a19Dec2017} and Remark \ref{rem-bac} in the case of
  the group action on $S\times\Theta$,
  $f_X(T,\theta)$ is $\mathcal O$-measurable
  for any $X\in\mathcal L^1(\ObservedData)$.
  Hence, $\mathcal O$ is of interest with respect to
  $(\ObservedData, \sigma(T,\theta))$.
  \qed
\end{proof}

Although the following lemmas may be well known, we could not find a proof in the literature.
Therefore, we present a proof here. We will use these lemmas in the next section.

\begin{lemma}\label{b18Jan2018}
  Let a measurable function $\theta:\Omega\rightarrow\Theta$ be surjective
  and $\mathcal G$ be a sub $\sigma$-algebra of
  $\sigma(\theta):=\{\theta^{-1}B\vert B\in\mathcal B(\Theta)\}$.
  Then, $\theta\mathcal G:=\{\theta(B)\vert B\in\mathcal G\}$
  is a sub $\sigma$-algebra of $\mathcal B(\Theta)$.
\end{lemma}
\begin{proof}\rm
  Since $\theta$ is surjective,
  $\Theta=\theta(\Omega)$ is an element of $\theta\mathcal G$.

  Let $A\in\theta\mathcal G$.
  There exists $B\in\mathcal G$ such that $A=\theta(B)$.
  By $\mathcal G\subset\sigma(\theta)$,
  there exists $C\in\mathcal B(\Theta)$ such that $B=\theta^{-1}C$.
  Since surjectivity of $\theta$ implies
  that $\theta(\theta^{-1}S)=S$ holds for any $S\subset\Theta$,
  we have
  $
  A=\theta(B)=\theta(\theta^{-1}C)=C.
  $
  By $\theta^{-1}A=\theta^{-1}C=B\in\mathcal G$,
  we have $\theta^{-1} A^c=(\theta^{-1}A)^c\in\mathcal G$.
  By surjectivity of $\theta$,
  $A^c=\theta(\theta^{-1}A^c)$ is an element of $\theta\mathcal G$.

  Suppose $A_n\in\theta\mathcal G$ for $n\in\mathbf N$.
  Analogously, we have $\theta^{-1}A_n\in\mathcal G$.
  Consequently,
  $
  \theta^{-1}\bigcup_{n\in\mathbf N}A_n
  =\bigcup_{n\in\mathbf N} \theta^{-1}A_n
  \in\mathcal G
  $
  implies
  $
  \bigcup_{n\in\mathbf N}A_n
  =\theta\(\theta^{-1}\bigcup_{n\in\mathbf N}A_n\)
  \in\theta\mathcal G.
  $
  \qed
\end{proof}

\begin{lemma}\label{d19Dec2017}
  Suppose that
  a measurable function $\theta:\Omega\rightarrow\Theta$ is surjective.
  Let $f_\lambda:\Theta\rightarrow\mathbf R\,(\lambda\in\Lambda)$ be
  measurable functions.
  Then, we have
  \begin{equation}\label{a18Jan2018}
  \sigma\(f_\lambda\circ\theta:\lambda\in\Lambda\)
  =\theta^{-1}\sigma\(f_\lambda:\lambda\in\Lambda\),
  \end{equation}
  where $\sigma\(f_\lambda:\lambda\in\Lambda\)$
  is the $\sigma$-algebra generated by
  $\{f_\lambda^{-1}B \vert \lambda\in\Lambda,\,B\in\mathcal B(\mathbf R)\}.$
\end{lemma}
\begin{proof}\rm
  Obviously, the right hand side of \eqref{a18Jan2018} includes
  the left hand side.
  We show the opposite inclusion.
  By the surjectivity of $\theta$, we have
  $
  f_\lambda^{-1}B
  =\theta\theta^{-1}f_\lambda^{-1}B
  =\theta(f_\lambda\circ\theta)^{-1}B
  \in\theta\sigma(f_\lambda\circ\theta:\lambda\in\Lambda)
  $
  for any $B\in\mathcal B(\mathbf R)$.
  By Lemma \ref{b18Jan2018},
  $\theta\sigma(f_\lambda\circ\theta:\lambda\in\Lambda)$
  is a sub $\sigma$-algebra of $\mathcal B(\Theta)$.
  Hence, we have
  \begin{equation}\label{c18Jan2018}
  \sigma(f_\lambda:\lambda\in\Lambda)
  \subset 
  \theta\sigma(f_\lambda\circ\theta:\lambda\in\Lambda).
  \end{equation}

  Note that we have 
  \begin{equation}\label{d18Jan2018}
  C=\theta^{-1}\theta C
  \quad (C\in\sigma(\theta)).
  \end{equation}
  In fact, since there exists $C'\in\mathcal B(\Theta)$ such that $C=\theta^{-1}C'$,
  we have
  $
  \theta^{-1}\theta C
  =\theta^{-1}\theta\theta^{-1}C'
  =\theta^{-1}C'
  =C.
  $

  Let $A\in\theta^{-1}\sigma(f_\lambda:\lambda\in\Lambda)$.
  There exists $B\in\sigma(f_\lambda:\lambda\in\Lambda)$
  such that $A=\theta^{-1}B$.
  By \eqref{c18Jan2018},
  there exists $C\in\sigma(f_\lambda\circ\theta:\lambda\in\Lambda)$
  such that $B=\theta^{-1}C$.
  Equation \eqref{d18Jan2018} and
  $
  \sigma(f_\lambda\circ\theta:\lambda\in\Lambda)
  \subset
  \sigma(\theta)
  $
  implies
  $A=\theta^{-1}\theta C=C\in\sigma(f_\lambda\circ\theta:\lambda\in\Lambda).$
  Consequently,
  the opposite inclusion holds.
  \qed
\end{proof}

\begin{lemma}\label{b9Jan2018}
  Let $V$ and $W$ be finite-dimensional vector spaces over $\mathbf R$,
  $V\oplus W$ be the direct sum of $V$ and $W$,
  $\pi:V\oplus W\rightarrow V$ be the projection,
  and $V^*$ be the dual space of $V$.
  Then, we have 
  \begin{equation}\label{a9Jan2018}
  \left\{
  B\in\mathcal B(V\oplus W)\vert
  B+W = B
  \right\}
  =
  \sigma(f\circ\pi : f\in V^*).
  \end{equation}
  Here, we put
  $
  B+W:=\{v+w \vert v\in B,\,w\in W\}.
  $  
\end{lemma}
\begin{proof}\rm
  Since $\pi^{-1}B+W=\pi^{-1}B$ holds for any $B\in\mathcal B(V)$,
  $\left\{ B\in\mathcal B(V\oplus W) \vert B+W = B\right\}$
  includes $\pi^{-1}\mathcal B(V)$.
  Let $\iota:V\rightarrow V\oplus W$ be the canonical injection.
  Suppose that $B\in\mathcal B(V\oplus W)$ satisfies
  $B+W = B$.
  Since we have
  \begin{align*}
    x\in \pi^{-1}\iota^{-1}B
    &\Leftrightarrow \iota\pi(x)\in B\\
    &\Leftrightarrow \pi(x)\in B
    &&\text{($\iota(y)=y$ holds for $y\in V$)}\\
    &\Rightarrow \pi(x)+(x-\pi(x))\in B+W
    &&\text{($x-\pi(x)\in W$)}\\
    &\Rightarrow x\in B+W\\
    &\Leftrightarrow x\in B
    &&\text{($B+W=B$)},
  \end{align*}
  $\pi^{-1}\iota^{-1}B\subset B$ holds.
  Since we can show the opposite inclusion analogously,
  we have $\pi^{-1}\iota^{-1}B=B$.
  By $\iota^{-1}B\in\mathcal B(V)$, $B$ is an element of $\pi^{-1}\mathcal B(V)$.
  Then, we have
  $$
  \left\{
  B\in\mathcal B(V\oplus W) \vert B+W = B
  \right\}
  =
  \pi^{-1}\mathcal B(V).
  $$
  
  Since $f\in V^*$ is a continuous map from $V$ to $\mathbf R$,
  $\mathcal B(V)$ includes $\sigma(f:f\in V^*)$.
  Let $\{f_1,\dots, f_n\}$ be a basis of $V^*$.
  Since any open subsets of $V\cong\mathbf R^n$ is a countable union
  of open sets of the form
  $$
  \bigcap_{i=1}^n f_i^{-1}\(\{x\in\mathbf R\vert a_i<x<b_i\}\)
  \quad (a_i, b_i\in\mathbf Q),
  $$
  we have
  $
  \mathcal B(V)
  \subset\sigma(f_1,\dots, f_n)
  \subset\sigma(f:f\in V^*).
  $
  Consequently, $\mathcal B(V)=\sigma(f:f\in V^*)$ holds
  and we have
  $$
  \left\{
  B\in\mathcal B(V\oplus W) \vert B+W = B
  \right\}
  =
  \pi^{-1}\sigma(f:f\in V^*).
  $$
  By Lemma \ref{d19Dec2017}, the right hand side of the above equation
  equals to $\sigma(f\circ\pi:f\in V^*)$.
  \qed
\end{proof}

\section{Application to the Conditional MLE problem} \label{sec:application}
In this section, we discuss a conditional MLE problem
for $\mathcal A$-distributions.

Let $A$ be an integer matrix of size $d\times n$,
and $b$ be an integer vector of length $n$.
Suppose that Poisson random variables
$X_k\sim\mathrm{Pois}(c_k),\,(k=1,\dots, n)$
are mutually independent.
We denote the conditional distribution of the random vector $X:=(X_1,\dots,X_n)^\top$
given $AX=b$ as an $\mathcal A$-distribution.
The parameters of $\mathcal A$-distribution are
$c=(c_1,\dots,c_n)^\top$ and $b=(b_1,\dots,b_n)^\top$.
The probability mass function of the $\mathcal A$-distribution
is given as
\begin{align*}
  \Prob\left(X=x\vert AX=b, \theta=c\right)
  & = \frac{\prod_{j=1}^n \frac{c_j^{x_j}}{x_j!} \exp\(-c_j\)}%
           {\sum_{Ay=b}\prod_{j=1}^n \frac{ c_j^{y_j}}{y_j!}\exp\(-c_j\)}
    = \frac{\prod_{j=1}^n \frac{c_j^{x_j}}{x_j!}}%
           {\sum_{Ay=b}\prod_{j=1}^n \frac{ c_j^{y_j}}{y_j!}}.
\end{align*}

An application of conditional distributions in statistics
is the elimination of nuisance parameters.
By Definition \ref{a14Nov2017} and Remark \ref{a22Nov2017},
the conditional distribution of a statistic
given the occurrence of a sufficient statistic of a nuisance parameter
does not depend on the value of the nuisance parameter.
This is an important property in similar tests and the Neyman--Scott problems
(see, e.g., \cite{amari2016information} and \cite{dojo}).
Hence, by the conditional distribution,
we can estimate the parameter of interest
without being affected by the nuisance parameter.
From this perspective,
we can regard the $\mathcal A$-distribution
as the conditional distribution given the sufficient statistic $AX$,
and the nuisance parameter corresponding to $AX$ is $A\theta$.
The traditional definition does not offer a mathematically clear
description of the parameter of interest for this case.
This is the motivation for the discussions in the previous section.
The space of parameters of interest is naturally described as
a sub $\sigma$-algebra under less restrictive conditions on $\theta$ and $c$.

The parameter $c$ of $\mathcal A$-distribution moves on
the set $\Theta:=\mathbf R_{\geq 0}^n$.
Consider the action of the multiplicative group $G:=\mathbf R_{>0}^d$
on the space $\Theta$ defined as
$$
g\cdot c
=\(c_j\prod_{i=1}^d g_i^{a_{ij}}\)_{j=1,\dots,n}
\quad (g\in G,\,c\in\Theta).
$$
This group action on $\Theta$ induces group action on
$\mathbf Z_{\geq 0}^d\times\Theta$ by
$$
g\cdot(b,c) = (b,g\cdot c)
\quad (g\in G,\, (b,c)\in \mathbf Z_{\geq 0}^d\times\Theta).
$$
For a vector $v=(v_1,\dots, v_n)^\top\in\mathbf R^n$, we use
$J(v)$ to denote the set of subscript $j$ that satisfies $v_j\neq 0$.
We also use $|J(v)|$ to denote the number of elements in $J(v)$.
Applying Lemma \ref{a22Dec2017}
in the case where $\ObservedData=\sigma(X)$,
$S=\mathbf Z_{\geq 0}^d$, and $T=AX$,
we have the following theorem:
\begin{thm}\label{b25Dec2017}
  The sub $\sigma$-algebra
  $$
  \mathcal O
  :=\{
   \{(AX,\theta)\in B\} \,\vert\,
   B\in \mathcal B(\mathbf Z_{\geq 0}^d)\times\mathcal B(\Theta),\,
   G\cdot B=B
  \}
  $$
  is of interest with respect to $(\sigma(X),\sigma(AX,\theta))$. 
\end{thm}
\begin{proof}\rm
  For any $g\in G$, we have
  \begin{align*}
    \Prob\left(X=x\vert AX=b, \theta=g\cdot c\right)
    &=
    g\cdot
    \frac{\prod_{j\in J(c)} c_j^{x_j}/x_j!}
         {\sum_{Ay=b}\prod_{j\in J(c)} c_j^{y_j}/y_j!}\\
    &=\frac{\prod_{j\in J(c)} \(c_j^{x_j}\prod_{i=1}^dg_i^{a_{ij}x_j}\)/x_j!}
         {\sum_{Ay=b}\prod_{j\in J(c)} \(c_j^{y_j}\prod_{i=1}^dg_i^{a_{ij}y_j}\)/y_j!}\\
    &=\frac{\prod_{i=1}^dg_i^{b_i}\prod_{j\in J(c)} c_j^{x_j}/x_j!}
         {\sum_{Ay=b}\prod_{i=1}^dg_i^{b_i}\prod_{j\in J(c)}c_j^{y_j}/y_j!}\\
    &=\frac{\prod_{j\in J(c)}c_j^{x_j}/x_j!}
         {\sum_{Ay=b}\prod_{j\in J(c)}c_j^{y_j}/y_j!}\\
    &= \Prob\left(X=x\vert AX=b, \theta=c\right).
  \end{align*}
  Since the conditional distribution of $X$ with respect to $(AX,\theta)$
  is invariant under the action of $G$ on $\mathbf Z_{\geq 0}^d\times\Theta$,
  for any $Y\in\mathcal L^1(\sigma(X))$,
  the conditional expectation
  $\E{Y\vert AX,\theta}$ is also invariant under the action.
  By Lemma \ref{a22Dec2017},
  $\mathcal O$ is of interest with respect to $(\sigma(X),\sigma(AX,\theta))$. 
  \qed
\end{proof}
Note that the quotient space $\Theta/G$ by the group action $G$ is not a manifold.
Therein lies the difficulty with describing the space of parameters of interest
and hence why we utilized the notion of $\sigma$-algebra of interest.

For $\alpha=(\alpha_1,\dots,\alpha_n)^\top\in\mathbf R^n$,
let $R_\alpha$ be the function from $\Theta=\mathbf R_{\geq 0}^n$
to $\mathbf R$ defined by
$$
R_\alpha(c):=
\begin{cases}
  \prod_{j\in J(\alpha)} c_j^{\alpha_j}
  &(\text{$c_j\neq 0$ for all $j\in J(\alpha)$})\\
  0
  &(\text{$c_j= 0$ for some $j\in J(\alpha)$})
\end{cases}
\quad (c=(c_1,\dots,c_n)^\top\in\Theta).
$$
Let $Z:\Theta\rightarrow\mathbf R^n$ be the function
defined by $Z(c):=(Z_1(c),\dots,Z_n(c))^\top\,(c\in\Theta)$ 
where
$$
Z_j(c):=
\begin{cases}
  1 & (c_j>0)\\
  0 & (c_j=0).
\end{cases}
$$
\begin{lemma}\label{f7Jan2018}
The random variables $AX$, $R_\alpha(\theta)\,(\alpha\in\ker A)$, and $Z(\theta)$
are $\mathcal O$-measurable.
\end{lemma}
\begin{proof}\rm
  Obviously, $AX$ is $\mathcal O$-measurable.
  Let $\pi:\mathbf Z_{\geq 0}^d\times\Theta\rightarrow\Theta$ be the projection.
  By some calculations, 
  we have
  \begin{align*}
    R_\alpha\circ\pi(g\cdot (t,c)) &= R_\alpha\circ\pi((t,c)),\, &
    Z\circ\pi(g\cdot (t,c)) &= Z\circ\pi((t,c))
  \end{align*}
  for any $\alpha\in\ker A$, $g\in G$,
  and $(t,c)\in\mathbf Z_{\geq 0}^d\times\Theta$.
  Consequently, 
  the functions $R_\alpha\circ\pi$ and $Z\circ\pi$
  are invariant under $G$.
  Applying Lemma \ref{a19Dec2017} in the case where $X=(AX,\theta)$,
  $R_\alpha\circ\pi(X)=R_\alpha(\theta)$ and
  $Z\circ\pi(X)=Z(\theta)$ are $\mathcal O$-measurable.
  \qed
\end{proof}

Let $\{e_1,\dots,e_n\}$ be the standard basis of $\mathbf R^n$,
i.e., the $i$-th component of $e_i$ is $1$ and the other components are $0$.
For the $d\times n$ matrix $A$, $\ker A$ and $\im A^\top$ can be written as
\begin{align*}
  \ker A&=
  \left\{\sum_{j=1}^nx_je_j\vert\sum_{j=1}^n a_{ij}x_j=0\right\},&
  \im A^\top&=
  \sum_{i=1}^d\mathbf R\sum_{j=1}^n a_{ij}e_j,
\end{align*}
where $a_{ij}$ is the $(i,j)$-component of $A$.
For $z\in\{0,1\}^n$,
let $\mathbf R^{J(z)}:=\sum_{j\in J(z)}\mathbf R e_j$
be the sub vector space of $\mathbf R^n$ spanned by $e_j\,(j\in J(z))$,
$
p_z:
\mathbf R^n\rightarrow\mathbf R^{J(z)}\,
\(\sum_{j=1}^nx_je_j\mapsto \sum_{j\in J(z)}x_je_j\)
$
be the projection,
and 
$\hat\iota_z:\mathbf R^{J(z)}\rightarrow\mathbf R^n$
be the canonical injection.
For $\alpha\in\mathbf R^n$,
we denote by $L_\alpha$ the linear map from $\mathbf R^n$ to $\mathbf R$
defined by
$$
L_\alpha(x)
= \sum_{j=1}^n \alpha_jx_j
\quad \(x=(x_1,\dots, x_n)^\top\in\mathbf R^n\).
$$
The compositions of $L_\alpha$ and $\hat\iota_z$ generate
a sub $\sigma$-algebra of $\mathcal B(\mathbf R^{J(z)})$,
and we put it as 
$
\sigma(L_\alpha\hat\iota_z:\alpha\in\mathbf R^{J(z)}\cap\ker A)
= \{(L_\alpha\hat\iota_z)^{-1}(B) \vert B\in\mathcal B(\mathbf R) \}.
$
\begin{lemma}\label{b6Jan2018}
  Under the same notation as above,
  the following equation holds for any $z\in\{0,1\}^n$:
  \begin{equation}\label{f18Jan2018}
  \left\{B\in\mathcal B(\mathbf R^{J(z)}) \vert
  B + p_z\im A^\top = B\right\}
  =
  \sigma(L_\alpha\hat\iota_z:\alpha\in\mathbf R^{J(z)}\cap\ker A).
  \end{equation}
\end{lemma}
\begin{proof}\rm
  With a map
  $$
  \langle \cdot,\cdot\rangle:
  \mathbf R^{J(z)}\times\mathbf R^{J(z)}\rightarrow\mathbf R
  \quad
  \(
  \left\langle \sum_{j\in J(z)}x_je_j,\, \sum_{j\in J(z)}y_je_j,\right\rangle
  =\sum_{j\in J(z)} x_jy_j
  \),
  $$
  $\mathbf R^{J(z)}$ is an inner product space.
  By the equation
  $$
  \mathbf R^{J(z)}\cap\ker A
  =
  \{
  v\in\mathbf R^{J(z)} \vert
  \text{$\langle v,w\rangle=0$ for all $w\in p_z\im A^\top$}
  \},
  $$
  we have
  $$
  \mathbf R^{J(z)}=
  \(\mathbf R^{J(z)}\cap\ker A\)
  \oplus
  p_z\im A^\top.
  $$
  By Lemma \ref{b9Jan2018},
  we have
  $$
  \left\{B\in\mathcal B(\mathbf R^{J(z)}) \vert
  B + p_z\im A^\top = B\right\}
  =
  \sigma\(f\circ\pi: f\in\(\mathbf R^{J(z)}\cap\ker A\)^*\),
  $$
  where $\pi:\mathbf R^{J(z)}\rightarrow\mathbf R^{J(z)}\cap\ker A$
  is the projection and
  $\(\mathbf R^{J(z)}\cap\ker A\)^*$ denotes the dual space
  of $\mathbf R^{J(z)}\cap\ker A$.
  Since we have 
  $$
  \{ f\circ\pi: f\in \(\mathbf R^{J(z)}\cap\ker A\)^* \}
  =
  \{L_\alpha\hat\iota_z:\alpha\in\mathbf R^{J(z)}\cap\ker A)\},
  $$
  Equation \eqref{f18Jan2018} holds.
    \qed
\end{proof}

\begin{lemma}\label{g6Jan2018}
  For $z\in\{0,1\}^n$,
  Let $\iota_z$ be the canonical injection
  from
  $\Theta_z:=\{c\in\Theta\vert Z(c)=z\}$
  to
  $\Theta$.
  Then, the inclusion
  \begin{equation}\label{e7Jan2017}
    \iota_z^{-1}\mathcal O^{**}\subset \iota_z^{-1}\sigma(R_\alpha:\alpha\in\ker A)
  \end{equation}
  holds.
  Here, we put $\mathcal O^{**}:=\{B\in\mathcal B(\Theta)\vert G\cdot B=B\}$.
\end{lemma}
\begin{proof} \rm
  Fix $z=(z_1,\dots,z_n)^\top\in\{0,1\}^n$.
  
  Let $B\in\mathcal O^{**}$.
  Since $\iota_z$ is a continuous map,
  $\iota_z^{-1}B$ is a Borel set of $\Theta_z$.
  For any $c\in\iota_z^{-1}B$ and $g\in G$,
  $\iota_z^{-1}B\subset\Theta_z$ and $G\cdot\Theta_z=\Theta_z$ implies
  $g\cdot c\in\Theta_z$.
  Since the equation 
  $
  \iota_z(g\cdot c)  = g\cdot c  \in G\cdot B = B
  $
  implies $g\cdot c\in\iota_z^{-1}B$,
  we have $G\cdot \iota_z^{-1}B=\iota_z^{-1}B$.
  Hence, we have the following inclusion relation:
  \begin{equation}\label{a05Jan2018}
  \iota_z^{-1}\mathcal O^{**}
  \subset
  \left\{B\in\mathcal B(\Theta_z) \vert G\cdot B = B\right\}.
  \end{equation}

  Suppose that $B\in\mathcal B(\Theta_z)$ satisfies $G\cdot B=B$.
  Let $\psi:\Theta_z\rightarrow\mathbf R^{J(z)}$ be a diffeomorphism
  defined by
  $
  \psi(c):=\sum_{j\in J(z)}\log(c_j)e_j\,
  (c=(c_1,\dots,c_n)^\top\in\Theta_z).
  $
  Note that $c_j$ is strictly positive for $j\in J(z)$.
  Since $\psi$ is a diffeomorphism, $\psi(B)$ is a Borel set of $\mathbf R^{J(z)}$.
  Put $v_i:=\sum_{j\in J(z)}a_{ij}e_j\in\mathbf R^{J(z)}\,(i=1,\dots, d)$.
  For any $c\in B$ and any $g_i\in\mathbf R\,(i=1,\dots,d)$,
  $\psi(c)+\sum_{i=1}^dg_iv_i$ is an element of $\psi(B)$.
  In fact, the inverse 
  \begin{align*}
    \psi^{-1}\(\psi(c)+\sum_{i=1}^dg_iv_i\)
    &=\psi^{-1}\(\sum_{j\in J(z)}\(\log(c_j)+\sum_{i=1}^da_{ij}g_i\)e_j \)
     =\sum_{j\in J(z)}\exp\(\log(c_j)+\sum_{i=1}^da_{ij}g_i\)e_j\\
    &=\sum_{j\in J(z)}\(c_j\prod_{i=1}^d \exp(g_ia_{ij})\)e_j
     = g'\cdot c
  \end{align*}
  is an element of $G\cdot B = B$.
  Here, we put $g':=\(\exp(g_i)\)_{i=1,\dots, d}\in G$.
  This implies $\psi(B)+p_z\im A^\top=\psi(B)$.
  Since $\psi^{-1}\psi(B)=B$ holds, we have 
  \begin{equation}\label{a6Jan2018}
    \left\{B\in\mathcal B(\Theta_z) \vert G\cdot B = B\right\}
    \subset
    \psi^{-1}
    \left\{B\in\mathcal B(\mathbf R^{J(z)}) \vert
    B + p_z\im A^\top = B\right\}.
  \end{equation}

  By Lemma \ref{b6Jan2018}, we have 
  \begin{equation}\label{a7Jan2018}
    \psi^{-1}
    \left\{B\in\mathcal B(\mathbf R^{J(z)}) \vert
    B + p_z\im A^\top = B\right\}
    =
    \psi^{-1}
    \sigma(L_\alpha\hat\iota_z:\alpha\in\mathbf R^{J(z)}\cap\ker A)
  \end{equation}

  Since the standard exponential mapping
  $\exp:\mathbf R\rightarrow\mathbf R_{>0}\,(x\mapsto\exp(x))$
  is a diffeomorphism, we have 
  $$
  \sigma(L_\alpha\hat\iota_z:\alpha\in\mathbf R^{J(z)}\cap\ker A)
  =
  \sigma(\exp L_\alpha\hat\iota_z:\alpha\in\mathbf R^{J(z)}\cap\ker A).
  $$

  The mappings $R_\alpha\iota_z\,(\alpha\in\mathbf R^{J(z)}\cap\ker A)$
  induce a $\sigma$-algebra on $\Theta_z$ as
  $$
  \sigma(R_\alpha\iota_z:\alpha\in\mathbf R^{J(z)}\cap\ker A)
  :=
  \sigma\(
  (R_\alpha\iota_z)^{-1}(B) :
  \alpha\in\ker A\cap \mathbf R^{J(z)},\,  B\in\mathcal B(\mathbf R)
  \).
  $$

  By Lemma \ref{d19Dec2017} and
  the equation $\exp L_\alpha\hat\iota_z\psi=R_\alpha\iota_z$,
  we have
  \begin{equation}\label{b7Jan2018}
  \psi^{-1}
  \sigma(\exp L_\alpha\hat\iota_z:\alpha\in\mathbf R^{J(z)}\cap\ker A)
  =
  \sigma(R_\alpha\iota_z:\alpha\in\mathbf R^{J(z)}\cap\ker A)
  \end{equation}

  Obviously, we have
  \begin{equation}\label{c7Jan2018}
    \sigma(R_\alpha\iota_z:\alpha\in\mathbf R^{J(z)}\cap\ker A)
    \subset 
    \sigma(R_\alpha\iota_z:\alpha\in\ker A).
  \end{equation}

  For any $B\in\mathcal B(\mathbf R)$ and any $\alpha\in\ker A$,
  $(R_\alpha\iota_z)^{-1}(B)=\iota_z^{-1}R_\alpha^{-1}(B)$
  is an element of $\iota_z^{-1}\sigma(R_\alpha:\alpha\in\ker A)$.
  Since $\iota_z^{-1}\sigma(R_\alpha:\alpha\in\ker A)$ is
  a $\sigma$-algebra on $\Theta_z$,
  we have
  \begin{equation}\label{d7Jan2017}
  \sigma(R_\alpha\iota_z:\alpha\in\ker A)
  \subset
  \iota_z^{-1}\sigma(R_\alpha:\alpha\in\ker A)
  \end{equation}

  By \eqref{a05Jan2018}, \eqref{a6Jan2018}, \eqref{a7Jan2018}, \eqref{b7Jan2018},
  \eqref{c7Jan2018} and \eqref{d7Jan2017},
  we have \eqref{e7Jan2017}.
  \qed
\end{proof}

\begin{lemma}\label{b19Jan2018}
  The following equation holds:
  \begin{equation}\label{e18Jan2018}
  \mathcal O^{**}=\sigma\( R_\alpha,\, Z ; \alpha\in\ker A \).
  \end{equation}
\end{lemma}
\begin{proof}\rm
  Since $R_\alpha$ and $Z$ are invariant under the action $G$,
  Lemma \ref{a19Dec2017} implies 
  $\mathcal O^{**}\supset\sigma\( R_\alpha,\, Z ; \alpha\in\ker A \)$.
  Let $B\in\mathcal O^{**}$.
  For $z\in\{0,1\}^n$, put $B_z:=B\cap\Theta_z\in\mathcal O^{**}$.
  Then, we have $B=\bigcup_{z\in\{0,1\}^n}B_z$.
  Since Lemma \ref{g6Jan2018} implies
  $B_z=\iota_z^{-1}B\in\iota_z^{-1}\sigma\( R_\alpha,\, Z ; \alpha\in\ker A \)$,
  there exists $\hat B_z\in\sigma\( R_\alpha,\, Z ; \alpha\in\ker A \)$
  such that $B_z=\iota_z^{-1}\hat B_z$.
  By $B_z=\iota_z^{-1}\hat B_z=\hat B_z\cap\Theta_z$,
  we have $B_z\in\sigma\( R_\alpha,\, Z ; \alpha\in\ker A \)$.
  Consequently, $\sigma\( R_\alpha,\, Z ; \alpha\in\ker A \)$ includes $B$.
  Hence, we have \eqref{e18Jan2018}.
  \qed
\end{proof}

\begin{lemma}\label{c19Jan2018}
  Let $\pi:\mathbf Z_{\geq 0}^d\times\Theta\rightarrow\Theta$ and
  $\pi':\mathbf Z_{\geq 0}^d\times\Theta\rightarrow\mathbf Z_{\geq 0}^d$
  be the projections.
  Put
  $
  \mathcal O^*
  :=
  \{
  B\in\mathcal B(\mathbf Z_{\geq 0}^d\times\Theta)
  \vert G\cdot B=B
  \}.
  $
  Then, the following equation holds:
  \begin{equation}\label{a19Jan2018}
  \mathcal O^*=\sigma\( \pi',\,R_\alpha\circ\pi,\, Z\circ\pi ; \alpha\in\ker A \).
  \end{equation}
\end{lemma}
\begin{proof}\rm
  Since $\pi'$, $R_\alpha\circ\pi$, and $Z\circ\pi$
  are invariant under the action $G$,
  Lemma \ref{a19Dec2017} implies 
  $$
  \mathcal O^*
  \supset
  \sigma\( \pi',\,R_\alpha\circ\pi,\, Z\circ\pi ; \alpha\in\ker A \).
  $$
  Let $B\in\mathcal O^*$.
  For $t\in\mathbf Z_{\geq 0}^d$,
  let $\iota_t:\Theta\rightarrow\mathbf Z_{\geq 0}^d\times\Theta$
  is an inclusion map defined by $\iota_t(c)=(t,c)$,
  and put $B_t:=B\cap\iota_t(\Theta)$.
  Then, $\iota_t^{-1}B_t$ is a Borel set of $\Theta$.
  Since the equation $G\cdot B=B$ implies $G\iota_t^{-1}B=\iota^{-1}B$,
  $\iota_t^{-1}B$ is an element of $\mathcal O^{**}$.
  By Lemma \ref{b19Jan2018} and Lemma \ref{d19Dec2017},
  we have
  $
  \pi^{-1}\iota_t^{-1}B_t
  \in\pi^{-1}\sigma\( R_\alpha,\, Z ; \alpha\in\ker A \)
  =\sigma\( R_\alpha\circ\pi,\, Z\circ\pi ; \alpha\in\ker A \).
  $
  Hence,
  $
  B_t
  =\pi^{-1}\iota_t^{-1}B_t\cap\iota_t(\Theta)
  \in\sigma\( \pi',\,R_\alpha\circ\pi,\, Z\circ\pi ; \alpha\in\ker A \)
  $
  implies
  $
  B=\bigcup_{t\in\mathbf Z_{\geq 0}^d}B_t
  \in\sigma\( \pi',\,R_\alpha\circ\pi,\, Z\circ\pi ; \alpha\in\ker A \).
  $
  We have \eqref{a19Jan2018}.
  \qed
\end{proof}

\begin{thm} \label{th:27Dec2017}
  Let $\hat\theta:\Omega\rightarrow\mathbf Z_{\geq 0}^d\times\Theta$
  be the measurable function defined by
  $\hat\theta(\omega)=(AX(\omega),\theta(\omega))$.
  If $\hat\theta$ is surjective,
  then the equation
  \begin{equation}\label{a27Dec2017}
  \mathcal O = \sigma\( AX, R_\alpha(\theta),\, Z(\theta) ; \alpha\in\ker A \)
  \end{equation}
  holds.
\end{thm}
\begin{proof} \rm
  By Lemma \ref{d19Dec2017} and Lemma \ref{c19Jan2018},
  we have
  \begin{align*}
  \mathcal O
  &=\hat\theta^{-1}\mathcal O^*
  =\hat\theta^{-1}\sigma\(\pi', R_\alpha\circ\pi,\, Z\circ\pi ; \alpha\in\ker A \)
  \\&
  =\sigma\(\pi'(\hat\theta), R_\alpha\circ\pi(\hat\theta),\, Z\circ\pi(\hat\theta)
  ; \alpha\in\ker A \)
  =\sigma\(AX,\, R_\alpha(\theta),\, Z(\theta) ; \alpha\in\ker A \).
  \end{align*}
  \qed
\end{proof}
This theorem implies that sub $\sigma$-algebra of interest $\mathcal O$ stands for
generalized odds ratios, which are, intuitively, parameters of interest.
Note that the parameter may lie on the border $\theta_i$.

As an interesting and important case of $\mathcal A$-distributions,
we consider the $r_1 \times r_2$ contingency table.
Let $u_{ij}$ be independent Poisson random variables
with parameter $\theta_{ij}\geq 0 \,(1\leq i\leq r_1,\,1\leq j\leq r_2)$.
The parameter $\theta:=(\theta_{ij})$ lies on
the set $\Theta:=\mathbf R_{\geq 0}^{r_1\times r_2}$.
As in the previous section, we regard $\theta$ as a measurable function
from $(\Omega,\mathcal F)$
to $(\Theta,\mathcal B(\Theta))$.
Note that we can assume that $\theta$ is surjective
without loss of generality.
Let $\ObservedData$ be the sub $\sigma$-algebra generated by all $u_{ij}$,
and $\mathcal G$ be the sub $\sigma$-algebra generated by
\begin{align*}
  \theta_{ij} &\,(1\leq i\leq r_1,\,1\leq j\leq r_2), &
  \sum_{i=1}^{r_1} u_{ij} &\,(1\leq j\leq r_2), &
  \sum_{j=1}^{r_2} u_{ij} &\,(1\leq i\leq r_1).
\end{align*}
For all $X\in\mathcal L^1(\ObservedData)$,
the conditional expectation $\E{X\vert\mathcal G}$ is invariant
under the action of the multiplicative group $G:=\mathbf R_{>0}^{r_1+r_2}$
on $\Theta$
defined by
$$
g\cdot c
:=
\(g_ig_{r_1+j}c_{ij}\)
\quad \(
 g=(g_i)\in G,\,
 c=(c_{ij})\in\Theta
\).
$$
For $1\leq i, k\leq r_1$ and $1\leq j, \ell\leq r_2$, 
let $R_{ijk\ell}:\Theta\rightarrow\mathbf R$ be a function
defined by 
$$
R_{ijk\ell}(c):=
\begin{cases}
  \frac{c_{ij}c_{k\ell}}{c_{i\ell}c_{kj}}
  & (c_{ij}c_{k\ell}c_{i\ell}c_{kj}\neq 0)\\
  0
  & (c_{ij}c_{k\ell}c_{i\ell}c_{kj}= 0)\\
\end{cases}
\quad (c=(c_{ij}) \in\Theta).
$$
Note that $R_{ijk\ell}$ is a function obtained from the odds ratio.
For $1\leq i \leq r_1$ and $1\leq j\leq r_2$,
we define a function $Z_{ij}:\Theta\rightarrow\mathbf R$ 
by
$$
Z_{ij}(c):=
\begin{cases}
  1 & (c_{ij}>0)\\
  0 & (c_{ij}=0)
\end{cases}
\quad (c=(c_{ij}) \in\Theta).
$$
The functions $Z_{ij}\,(1\leq i\leq r_1,\,1\leq j\leq r_2)$
hold information on the position of zero cells.
The functions $R_{ijk\ell}$ and $Z_{ij}$ are invariant
with respect to the action of group $G$.
By Lemma \ref{a19Dec2017}, random variables $R_{ijk\ell}(\theta)$ and $Z_{ij}(\theta)$ are $\mathcal O$-measurable.

The following theorem states that $A\theta$ is a nuisance parameter.
\begin{thm}\label{th:29May2018}
The following equation holds:
$$
\sigma(AX, \theta) = \sigma(A\theta, \mathcal O).
$$
\end{thm}

\begin{cor}
  $\sigma(A\theta)$ is nuisance for $(\sigma(X),\mathcal O)$.
\end{cor}
\begin{proof} \rm
By Theorem \ref{b25Dec2017},
for any $Y\in\mathcal L^1(\sigma(X))$,
$\E{Y\vert \sigma\(AX,\theta\)}$ is $\mathcal O$-measurable.
The equation in Therorem \ref{th:29May2018} implies
$
\E{Y\vert \sigma\(AX,\theta\)}
=
\E{Y\vert \sigma\(A\theta,\mathcal O\)}.
$
Hence, $\mathcal O$ is of interest with respect to
$(\sigma(X),\sigma\(A\theta,\mathcal O\))$.
Therefore $\sigma(A\theta)$ is nuisance for $(\sigma(X),\mathcal O)$.
\qed
\end{proof}

To show Theorem \ref{th:29May2018}, we prepare the following lemma:
\begin{lemma}\label{e24May2018}
Let $F:\mathbf R^n\rightarrow\mathbf R^d$ be a linear map
and $\iota:\mathbf R_{>0}^n\rightarrow\mathbf R^n$ be the inclusion.
For $\alpha\in\ker F$, let $R_\alpha:\mathbf R_{>0}^n\rightarrow\mathbf R $
be a function defined by $R_\alpha(x):= \prod_{i=1}^n x_i^{\alpha_i}$.
Then, we have
$$
\mathcal B(\mathbf R_{>0}^n)
=
\sigma(F\iota, R_\alpha ; \alpha\in\ker F).
$$
\end{lemma}
\begin{proof} \rm
It is ovbious that the left-hand side includes
the right-hand side.
We show the opposite inclusion.
Let $\{\alpha_1,\dots, \alpha_m\}$ be a basis of $\ker F$.
Then the differential map 
$$
\varphi:\mathbf R_{>0}^n\rightarrow (\mathrm{Im}F\iota)\times\mathbf R_{>0}^m
\quad
(x\mapsto (F\iota(x),\, R_{\alpha_1}(x),\dots,R_{\alpha_m}(x) )  )
$$
is surjective.
By the general theory of the exponential family
(\cite[p.~125]{TKT}, \cite[Theorem 3.~6]{LawrenceDBrown1986}),
$\varphi$ is also injective.
Hence, $\varphi$ is a diffeomorphism
between $\mathbf R_{>0}^n$ and $(\mathrm{Im}F\iota)\times\mathbf R_{>0}^m$,
and we have
\begin{align*}
\mathcal B(\mathbf R_{>0}^n)
&
= \varphi^{-1}\mathcal B\( (\mathrm{Im}F\iota)\times\mathbf R_{>0}^m\)
= \varphi^{-1} \sigma\(p, q_1,\dots, q_m\)
= \sigma\(p\varphi, q_1\varphi,\dots, q_m\varphi\)
\\&
= \sigma\(F\iota,\, R_{\alpha_1}(x),\dots,R_{\alpha_m}\)
\subset \sigma(F\iota, R_\alpha ; \alpha\in\ker F)
\end{align*}
Here,
$p:(\mathrm{Im}F\iota)\times\mathbf R_{>0}^m\rightarrow (\mathrm{Im}F\iota)$
and
$q_i:(\mathrm{Im}F\iota)\times\mathbf R_{>0}^m\rightarrow \mathbf R_{>0}$
are the projections.
\qed
\end{proof}

\begin{proof}{\bf of Theorem \ref{th:29May2018}} \rm
Recall that, for $z\in\{0,1\}^n$, we put
$
\Theta_z = \{c\in\Theta\,\vert\, Z(c)=z \}
$
and that $\iota_z:\Theta_z\rightarrow\Theta$ is the inclusion.
Applying Lemma \ref{e24May2018} in the case where $F=A\iota_z$,
we have
\begin{equation}\label{a24May2018}
\mathcal B(\Theta_z)
= 
\sigma(A\iota_z, R_\alpha\iota_z ; \alpha\in\ker A\iota_z).
\end{equation}
The equation $\ker A\iota_z=\mathbf R^{J(z)}\cap \ker A$
implies 
\begin{equation}\label{b24May2018}
\sigma\(A\iota_z,\,R_\alpha\iota_z: \alpha\in\ker A\iota_z\)
=
\sigma\(A\iota_z,\,R_\alpha\iota_z: \alpha\in\mathbf R^{J(z)}\cap \ker A\)
\end{equation}
By Equations \eqref{a24May2018} and \eqref{b24May2018}, we have
\begin{align}
\mathcal B(\Theta_z)
&=
\sigma\(A\iota_z,\,R_\alpha\iota_z: \alpha\in\mathbf R^{J(z)}\cap \ker A\)
\subset
\iota_z^{-1}\sigma\(A,\,R_\alpha: \alpha\in\mathbf R^{J(z)}\cap \ker A\)
\nonumber\\&
\subset
\iota_z^{-1}\sigma\(A,\,R_\alpha,\,Z: \alpha\in\mathbf \ker A\)
\subset 
\sigma\(A,\,R_\alpha,\,Z: \alpha\in\mathbf \ker A\).
\label{c24May2018}
\end{align}
Any $B\in\mathcal B(\Theta)$ can be decomposed as 
$
B
= \bigcup_{z\in\{0,1\}^n} \(B\cap \Theta_z\)
= \bigcup_{z\in\{0,1\}^n} \iota_z^{-1}B.
$
By \eqref{c24May2018},
$B$ is an element of
$\sigma\(A,\,R_\alpha,\,Z: \alpha\in\mathbf \ker A\)$.
Hence, we have
\begin{equation}\label{d24May2018}
\mathcal B(\Theta)
=
\sigma\(A,\,R_\alpha,\,Z:\alpha\in\ker A\)
\end{equation}
The $\sigma$-algebra generated by $\theta$ is
the pull-back of 
the left-hand side of \eqref{d24May2018}
with respect to $\theta$.
By Lemma \ref{d19Dec2017}, 
the pull-back of the right-hand side of \eqref{d24May2018} equals to
$
\sigma\(A\theta,\,R_\alpha(\theta),\,Z(\theta):\alpha\in\ker A\).
$
Hence, we have 
$$
\sigma\(\theta\)
=
\sigma\(A\theta,\,R_\alpha(\theta),\,Z(\theta):\alpha\in\ker A\).
$$
This equation implies 
$$
\sigma\(AX,\,\theta\)
=
\sigma\(AX,\,A\theta,\,R_\alpha(\theta),\,Z(\theta):\alpha\in\ker A\)
$$
By Theorem 4, we have
\begin{align*}
\sigma\(A\theta,\mathcal O\)
&
=\sigma\(A\theta,\,\sigma\(AX,\,R_\alpha(\theta),\,Z(\theta):\alpha\in\ker A\)\)
\\&
=\sigma\(A\theta,\,AX,\,R_\alpha(\theta),\,Z(\theta):\alpha\in\ker A\)
=\sigma\(AX,\,\theta\)
\end{align*}
\qed
\end{proof}

\section{Examples of CMLE problems}


\def\mytheta{p}
Theorem \ref{th:27Dec2017} and \ref{th:29May2018} claim that when $AX$ is given,
$\sigma(R_\alpha(\theta), Z(\theta))$ are of interest and 
$\sigma(A\theta)$ is a nuisance.
In the case of contingency tables, generalized odds ratios $R_\alpha(\mytheta)$ and positions
of zero cells $Z(\mytheta)$ are of interest
and row and column probabilities $A\mytheta$ are a nuisance
when the marginal sums of the table are given.
We present examples of estimating generalized odds ratios
by CMLE.

\begin{example} \rm
We generate categorical data 
concerning the number of hours slept and time of going to bed
from a student sample in the LearnBayes package
\footnote{\url{https://cran.r-project.org/web/packages/LearnBayes/index.html}}
of the system R for statistical computing.

Rows are categorized by time spent sleeping.
The categories are sleeping
less than 6 hours, 6--7 hours, and more than 7 hours.
Columns are categorized by the time of going to bed.
The categories are going to bed
before midnight, between midnight and 1am, and after 1am.
We wish to analyze these categorical data by the Poisson random model
$U_{ij} \sim {\rm Pois}(\mytheta_{ij})$. 
The independence of rows and columns is rejected by the $\chi^2$ test
with the threshold $p$-value $0.05$.
Then, we regard the column sum $\sum_i \mytheta_{ij}$
and the row sum $\sum_j \mytheta_{ij}$ as nuisance parameters.
These represent probabilities of the event standing for $j$-th row and 
one standing for $i$-th column
when the rows and the columns are independent.
We perform CMLE under the condition that column sums $\sum_i  u_{ij}$ and
row sums $\sum_j u_{ij}$ are given.
\medbreak

\noindent
Categorical data for all:
\begin{center}
\begin{tabular}{r|ccc}
Bed time $\backslash$ Hours slept&  less than 6 hour & 6--7 & more than 7 hours \\ \hline
Before 24   & 1 & 6 & 123 \\
24--25  & 3 & 22 & 145 \\
After 25 & 86 & 91 & 176 \\
\end{tabular} 
\end{center}
We omit titles and express this table as 
$
\begin{pmatrix}
 1&  6&  123 \cr
 3&  22&  145 \cr
 86&  91&  176 \cr
\end{pmatrix}
$.
Categorical data for males:
$$
\begin{pmatrix}
 1&  2&  28 \cr
0&  4&  47 \cr
 35&  32&  71 \cr
\end{pmatrix}
$$
Categorical data for females:
$$
\begin{pmatrix}
0 &  4&  95 \cr
 3&  18&  98 \cr
 51&  59&  105 \cr
\end{pmatrix}
$$

\medbreak
Because this CMLE can be solved by the $\mathcal{A}$-distribution discussed previously,
we apply our algorithm for evaluating normalizing constants and their 
derivatives to the method for estimating the conditional maximum likelihood in \cite[\S 4]{TKT}. 
We obtain the following estimates.
\noindent
CMLE $(\mytheta_{ij})$ for all:
$$
\begin{pmatrix}
 0.176556059977815&  1&  10.5634953362788 \cr
 0.144532927997885&  1&  3.39969669537228 \cr
 1&  1&  1 \cr
\end{pmatrix}
$$
CMLE for males:
$$
\begin{pmatrix}
 0.458167657900967&  1&  \underline{6.25676090279981} \cr
0&  1&  \underline{5.25200491199345} \cr
 1&  1&  1 \cr
\end{pmatrix}
$$
CMLE for females:
$$
\begin{pmatrix}
0&  1&  \underline{13.2714773737657} \cr
 0.193351042187373&  1&  \underline{3.04872586155291} \cr
 1&  1&  1 \cr
\end{pmatrix}
$$
As explained in the previous section,
the space of parameters of interest should be regarded as 
the collection of different orbits by the torus action.
When the parameter value obtained via CMLE is $(\mytheta_{ij})$, 
values on the orbit
$(\gi_i \gj_j \mytheta_{ij})$, $\gi_i, \gj_j \in \mathbb{R}_{>0}$
are equivalent parameters.
Since the normalized elements of the second column and the third row are $1$,
we have
$\gi_3 \gj_1 = \gi_3 \gj_2 = \gi_3 \gj_3 = 1$
and
$\gi_1 \gj_2 = \gi_2 \gj_2 = \gi_3 \gj_2 = 1$.
Then, we have $\gi_i \gj_j =1$ for all $i, j$. 
The condition whereby this normalization is possible 
($\mytheta_{i2} \not= 0$, $\mytheta_{3j} \not= 0$)
defines a subspace of the parameters of interest.
The subspace is isomorphic to $\mathbb{R}_{\geq 0}^4$ by the quotient topology.
The correspondence is given by 
\begin{equation} \label{eq:odds1}
(\mytheta_{ij}) \mapsto
\left(
\begin{array}{ccc}
\frac{\mytheta_{11}\mytheta_{32}}{\mytheta_{12} \mytheta_{31}} & 1 & \frac{\mytheta_{13}\mytheta_{32}}{\mytheta_{12} \mytheta_{33}} \\
\frac{\mytheta_{21}\mytheta_{32}}{\mytheta_{22} \mytheta_{31}} & 1 & \frac{\mytheta_{23}\mytheta_{32}}{\mytheta_{22} \mytheta_{33}} \\
1 & 1 & 1
\end{array}
\right)
\end{equation} 
In this chart, males and females exhibit different tendencies.
For example, the underlined values at $(1,3)$ and $(2,3)$ positions are close in the case of males but not for females.

The number obtained by replacing $\mytheta_{ij}$ by the frequency $u_{ij}$
in (\ref{eq:odds1}) is called a generalized odds ratio.
Generalized odds ratios for our data are as follows.
Odds ratios for all:
$$
\begin{pmatrix}
 0.176356589147287&  1&  10.5994318181818 \cr
 0.144291754756871&  1&  3.40779958677686 \cr
 1&  1&  1 \cr
\end{pmatrix}
$$
Odds ratios for males:
$$
\begin{pmatrix}
 0.457142857142857&  1&  6.30985915492958 \cr
0&  1&  5.29577464788732 \cr
 1&  1&  1 \cr
\end{pmatrix}
$$
Odds ratios for females:
$$
\begin{pmatrix}
0&  1&  13.3452380952381 \cr
 0.19281045751634&  1&  3.05925925925926 \cr
 1&  1&  1 \cr
\end{pmatrix}
$$
Note that,
as proved in \cite[Theorem 5]{TKT}, these generalized odds ratios
approximate CMLE because we have a sufficient sample size.
\end{example}

When the sample size is relatively small,
a generalized odds ratio may not approximate the corresponding CMLE well.
We present one example.
\begin{example} \rm
The categorical data below are taken from emergency safety information
on diclofenac sodium for influenza encephalitis 
and encephalopathy\footnote{Pharmaceuticals and Medical Devices Agency, Japan, 2000, \url{https://www.pmda.go.jp/files/000148557.pdf}}. \\
Categorical data:
\begin{center}
\begin{tabular}{c|ccc}
        & acetaminophen & diclofenac sodium & mefenamic acid \\ \hline
death & 4                         &  7  & 2 \\ 
survival & 32                       &  5   & 6 \\ \hline
\end{tabular}
\end{center}
We omit titles and express this table as 
$
\left(
\begin{array}{ccc}
 4 & 7 & 2 \\
 32& 5 & 6 \\
\end{array}
\right)
$.
By applying our algorithm and the method in \cite{TKT}, we obtain
the following CMLE.
$$
\begin{pmatrix}
 1&  \underline{10.5557279737263} &  2.62096714359908 \cr
 1&  1&  1 \cr
\end{pmatrix}
$$
Generalized odds ratios are
$$
\begin{pmatrix}
 1&  \underline{11.2} &  2.66666666666667 \cr
 1&  1&  1 \cr
\end{pmatrix}
$$
See the numbers underlined above.
We observe that the odds ratio is larger than the CMLE.
In other words, the effect of nuisance parameters increases the risk
in this case.
Finally, we briefly note how subsequent data released from the same institute in 2001 appeared to show that diclofenac sodium was in fact more associated with survival, rather than death. This reminds us of some of the difficulties inherent in statistical analyses. Here are those new data: 
\footnote{\url{http://idsc.nih.go.jp/disease/influenza/iencepha.html}}.
\begin{center}
\begin{tabular}{c|ccc}
        & acetaminophen & diclofenac sodium & mefenamic acid \\ \hline
death & 23                         &  13  & 6 \\ 
survival & 78                       &  25   & 9 \\ \hline
\end{tabular}
\end{center}
\noindent
Our algorithm outputs
CMLE
$$
\begin{pmatrix}
 1&  1.7567483756645&  2.24788463785377 \cr
 1&  1&  1 \cr
\end{pmatrix}
$$
\noindent
and odds ratios:
$$
\begin{pmatrix}
 1&  1.76347826086957&  2.26086956521739 \cr
 1&  1&  1 \cr
\end{pmatrix} .
$$
\end{example}

\section{Appendix}\label{sec:appendix}
\small
We will explain 
the derivation of the matrix $U_2$ of Example \ref{example:5}
with twisted cohomology groups
by following \cite{MG} and the program  
{\tt gtt\_ekn3/ekn\_pfaffian\_8.rr} of the package {\tt gtt\_ekn3}.

We start with the integral representation of ${}_2 F_1$:
  $$
  \frac{\Gamma (b)\Gamma (c-b)}{\Gamma (c)} \cdot 
  {}_2 F_1 (a,b,c;x)
  =\int_0^1 t^{b-1}(1-t)^{c-b-1}(1-xt)^{-a} dt
  =(-1)^b \int_0^{-1} t^{b}(1+xt)^{-a}(1+t)^{c-b-1} \frac{dt}{t}.
  $$
We rename the parameters $a,b,c$ by
  $$
  (\alpha_0 ,\alpha_1 ,\alpha_2 ,\alpha_3)=(a-c+1,b,-a,c-b-1).
  \footnote{$\alpha_0=-\alpha_1-\alpha_2-\alpha_3$ stands for the exponent
at infinity.}
  $$
The decrement of $a$ stands for an increment of $\alpha_2$ 
(and decrement of $\alpha_0$).
The identity we want to derive is $F(a)=M(a)F(a+1)$,
which is a special case of 
  $$
  \mathbf{S}(\alpha ;x) =\frac{1}{\alpha_2 }U_2(\alpha_{(2)} ;x) \mathbf{S}(\alpha_{(2)} ;x) ,\quad 
  \alpha_{(2)}:=(\alpha_0 +1,\alpha_1 ,\alpha_2 -1,\alpha_3)
  $$
in \cite[Corollary 6.3]{MG} 
($\alpha_{(2)}$ stands for $a+1$ ).   
The function {\tt upAlpha(2,1,1)} in the program derives 
$\frac{1}{\alpha_2} U_2$.
$\mathbf{S}(\alpha ;x)$ is the vector consisting of
the hypergeometric series $S(\alpha ;x)$ 
defined in  \cite[Section 6]{MG}
and its derivatives (Gauss-Manin vector).
When  $c\in \mathbb{N}_0$,
it can be expressed in terms of ${}_2F_1$ as
  $$
  \mathbf{S}(\alpha ;x)=
  \begin{pmatrix}
    S \\ \frac{1}{\alpha_2}\theta_x S
  \end{pmatrix}
  =
  \begin{pmatrix}
    1 & 0 \\ 0 & 1/\alpha_2
  \end{pmatrix}
  \begin{pmatrix}
    S \\ \theta_x S
  \end{pmatrix}
  = \frac{1}{(-a)! (-b)! (c-1)!}
  \begin{pmatrix}
    1 & 0 \\ 0 & 1/\alpha_2
  \end{pmatrix}
  \begin{pmatrix}
    _2 F_1 \\ \theta_x {}_2 F_1
  \end{pmatrix}.
  $$
Hence, the matrix $M(a)$ can be expressed as
  $$
  M(a)=-a 
  \begin{pmatrix}
    1 & 0 \\ 0 & \alpha_2
  \end{pmatrix}
  \Big( \frac{1}{\alpha_2}U_2(\alpha_{(2)}) \Big)
  \begin{pmatrix}
    1 & 0 \\ 0 & 1/(\alpha_2-1)
  \end{pmatrix}
  = 
  \begin{pmatrix}
    1 & 0 \\ 0 & \alpha_2
  \end{pmatrix}
  U_2(\alpha_{(2)}) 
  \begin{pmatrix}
    1 & 0 \\ 0 & 1/(\alpha_2-1)
  \end{pmatrix}.
  $$
It follows from 
  \cite[Theorem 5.3]{MG} that the representation matrix  $U_2$ can
be expressed as
  $$
  U_2 (\alpha_{(2)};x)
  =C(\alpha)P_2 (\alpha)^{-1} D_2(x) Q_2 (\alpha_{(2)}) C(\alpha_{(2)})^{-1}.
  $$
We use the notation
$|\tilde{x}\langle ij \rangle |$,
which is the determinant of the minor matrix consisting
of the $i$-th column and the $j$-th column of the matrix
  $\tilde{x}=
  \begin{pmatrix}
    1&0&1&1 \\ 0&1&x&1
  \end{pmatrix}
$,
where the numbering starts with $0$
(see \cite{MG} as to details).
We put $\varphi \langle ij \rangle =\frac{|\tilde{x}\langle ij \rangle | dt}{L_i L_j}$, 
where $L_0=1$, $L_1=t$, $L_2=1+xt$, and $L_3=1+t$.
We have the following expressions with these notations.
  \begin{align*}
    D_2(x)&={\rm diag} \left( 
    \frac{|\tilde{x}\langle 21 \rangle|}{|\tilde{x}\langle 01 \rangle|} ,
    \frac{|\tilde{x}\langle 23 \rangle|}{|\tilde{x}\langle 03 \rangle|} \right)
    ={\rm diag}(1,1-x) 
    =
    \begin{pmatrix}
      1 & 0 \\ 0 &1-x
    \end{pmatrix},\\
    C(\alpha)&=
    \begin{pmatrix}
      \mathcal{I}(\varphi \langle 01 \rangle,\varphi \langle 01 \rangle ) & 
      \mathcal{I}(\varphi \langle 01 \rangle,\varphi \langle 02 \rangle ) \\ 
      \mathcal{I}(\varphi \langle 02 \rangle,\varphi \langle 01 \rangle ) & 
      \mathcal{I}(\varphi \langle 02 \rangle,\varphi \langle 02 \rangle )
    \end{pmatrix}
    =2\pi\sqrt{-1}
    \begin{pmatrix}
      \frac{1}{\alpha_0}+\frac{1}{\alpha_1} & \frac{1}{\alpha_0} \\
      \frac{1}{\alpha_0} & \frac{1}{\alpha_0}+\frac{1}{\alpha_2}
    \end{pmatrix} ,\\
    Q_2(\alpha)&=
    \begin{pmatrix}
      \mathcal{I}(\varphi \langle 01 \rangle,\varphi \langle 01 \rangle ) & 
      \mathcal{I}(\varphi \langle 01 \rangle,\varphi \langle 02 \rangle ) \\ 
      \mathcal{I}(\varphi \langle 03 \rangle,\varphi \langle 01 \rangle ) & 
      \mathcal{I}(\varphi \langle 03 \rangle,\varphi \langle 02 \rangle )
    \end{pmatrix}
    =2\pi\sqrt{-1}
    \begin{pmatrix}
      \frac{1}{\alpha_0}+\frac{1}{\alpha_1} & \frac{1}{\alpha_0} \\
      \frac{1}{\alpha_0} & \frac{1}{\alpha_0}
    \end{pmatrix} ,\\
    P_2(\alpha)&=
    \begin{pmatrix}
      \mathcal{I}(\varphi \langle 21 \rangle,\varphi \langle 01 \rangle ) & 
      \mathcal{I}(\varphi \langle 21 \rangle,\varphi \langle 02 \rangle ) \\ 
      \mathcal{I}(\varphi \langle 23 \rangle,\varphi \langle 01 \rangle ) & 
      \mathcal{I}(\varphi \langle 23 \rangle,\varphi \langle 02 \rangle )
    \end{pmatrix}
    =2\pi\sqrt{-1}
    \begin{pmatrix}
      \frac{1}{\alpha_1} & -\frac{1}{\alpha_2} \\
      0 & -\frac{1}{\alpha_2}
    \end{pmatrix},
  \end{align*}
where $\mathcal{I}$ is the intersection form on the twisted cohomology group. 
The inverse matrices of them can also be expressed in terms of
intersection numbers as in \cite[Appendix]{MG}.
This method is implemented as the function 
{\tt invintMatrix\_k}  in our package
and it outputs
  \begin{align*}
    P_2(\alpha)^{-1}
    &=\frac{1}{(2\pi\sqrt{-1})^2}
    \begin{pmatrix}
      \alpha_1 & 0 \\ 0 & \alpha_2
    \end{pmatrix} 
    \begin{pmatrix}
      \mathcal{I}(\varphi \langle 31 \rangle,\varphi \langle 01 \rangle ) & 
      \mathcal{I}(\varphi \langle 31 \rangle,\varphi \langle 03 \rangle ) \\ 
      \mathcal{I}(\varphi \langle 32 \rangle,\varphi \langle 01 \rangle ) & 
      \mathcal{I}(\varphi \langle 32 \rangle,\varphi \langle 03 \rangle )
    \end{pmatrix} 
    \begin{pmatrix}
      \alpha_1 & 0 \\ 0 & \alpha_3
    \end{pmatrix}\\
    &=\frac{1}{2\pi\sqrt{-1}}
    \begin{pmatrix}
      \alpha_1 & 0 \\ 0 & \alpha_2
    \end{pmatrix} 
    \begin{pmatrix}
      \frac{1}{\alpha_1} & -\frac{1}{\alpha_3} \\ 0 & -\frac{1}{\alpha_3}
    \end{pmatrix} 
    \begin{pmatrix}
      \alpha_1 & 0 \\ 0 & \alpha_3
    \end{pmatrix}
    =\frac{1}{2\pi\sqrt{-1}}
    \begin{pmatrix}
      \alpha_1 & -\alpha_1 \\ 0 & -\alpha_2
    \end{pmatrix}  ,\\ 
    C(\alpha)^{-1}
    &=\frac{1}{(2\pi\sqrt{-1})^2}
    \begin{pmatrix}
      \alpha_1 & 0 \\ 0 & \alpha_2
    \end{pmatrix} 
    \begin{pmatrix}
      \mathcal{I}(\varphi \langle 31 \rangle,\varphi \langle 31 \rangle ) & 
      \mathcal{I}(\varphi \langle 31 \rangle,\varphi \langle 32 \rangle ) \\ 
      \mathcal{I}(\varphi \langle 32 \rangle,\varphi \langle 31 \rangle ) & 
      \mathcal{I}(\varphi \langle 32 \rangle,\varphi \langle 32 \rangle )
    \end{pmatrix}     
    \begin{pmatrix}
      \alpha_1 & 0 \\ 0 & \alpha_2
    \end{pmatrix}\\
    &=\frac{1}{2\pi\sqrt{-1}}
    \begin{pmatrix}
      \alpha_1 & 0 \\ 0 & \alpha_2
    \end{pmatrix} 
    \begin{pmatrix}
      \frac{1}{\alpha_3}+\frac{1}{\alpha_1} & \frac{1}{\alpha_3} \\ 
      \frac{1}{\alpha_3} & \frac{1}{\alpha_3}+\frac{1}{\alpha_2}
    \end{pmatrix}
    \begin{pmatrix}
      \alpha_1 & 0 \\ 0 & \alpha_2
    \end{pmatrix}
    =\frac{\alpha_1 \alpha_2}{2\pi\sqrt{-1}\cdot \alpha_3}
    \begin{pmatrix}
      \frac{\alpha_1 +\alpha_3}{\alpha_2} & 1 \\ 
      1 & \frac{\alpha_2 +\alpha_3}{\alpha_1}
    \end{pmatrix}.
  \end{align*}
These matrices can be obtained in our program as
  \begin{align*}
    &D_2(x)={\tt repMatrix(2,1,1)},& 
    &C(\alpha)/(2\pi\sqrt{-1})={\tt intMatrix([0,3],[0,3],1,1)},& \\
    &P_2(\alpha)/(2\pi\sqrt{-1})={\tt intMatrix([2,0],[0,3],1,1)},&  
    &Q_2(\alpha)/(2\pi\sqrt{-1})={\tt intMatrix([0,2],[0,3],1,1)},& \\
    &(2\pi\sqrt{-1})P_2(\alpha)^{-1}={\tt invintMatrix\_k([2,0],[0,3],1,1)},& 
    &(2\pi\sqrt{-1})C(\alpha)^{-1}={\tt invintMatrix\_k([0,3],[0,3],1,1)}
  \end{align*}
  (the argument $(1,1)$ stands for $(r_1-1, r_2-1)$). 

\normalsize

\medbreak
{\it Acknowledgement}:
This work was supported by MEXT/JSPS KAKENHI Grant Numbers JP
25220001, 
17K05279, 
18J01507 
and by Research Institute for Mathematical Sciences,
a Joint Usage/Research Center located in Kyoto University.
We deeply appreciate several constructive criticisms by the reviewers,
which made big improvements of our algorithms and implementation.

\end{document}